\documentclass[a4paper,11pt]{article}
\usepackage{amsmath,amsthm,amssymb}
\usepackage{amscd}
\usepackage{mathrsfs}
\usepackage{graphicx, xcolor}
\usepackage{url}
\usepackage[all]{xy}
\usepackage{tikz}
\usepackage{ulem}
\usetikzlibrary{calc}
\usepackage{authblk}

\theoremstyle{definition}
\newtheorem{dfn}{Definition}[section]
\newtheorem{example}[dfn]{Example}

\newtheorem{remark}[dfn]{Remark}

\newtheorem*{claim}{Claim}

\theoremstyle{plain}
\newtheorem{prop}[dfn]{Proposition}
\newtheorem{thm}[dfn]{Theorem}
\newtheorem{lem}[dfn]{Lemma}

\newcommand{\Q}{\mathbb{Q}}
\newcommand{\gr}{\mathrm{gr}}
\newcommand{\std}{\mathrm{std}}
\newcommand{\pa}{\partial}
\newcommand{\grt}{{\sf grt}}
\newcommand{\GRT}{{\sf GRT}}
\newcommand{\tder}{{\sf tder}}
\newcommand{\sder}{{\sf sder}}
\newcommand{\krv}{{\sf krv}}
\newcommand{\KRV}{{\sf KRV}}
\newcommand{\dk}{{\sf dk}}
\newcommand{\edk}{{\sf edk}}
\newcommand{\lie}{{\sf lie}}
\newcommand{\ass}{{\sf ass}}
\newcommand{\tr}{{\sf tr}}
\newcommand{\EM}{\mathrm{em}}

\def\bbQ{{\mathbb Q}}
\def\bbR{{\mathbb R}}
\def\calA{{\mathcal A}}

\def\calK{{\mathcal K}}

\def\dk{{\mathsf{dk}}}
\def\edk{{\mathsf{edk}}}

\def\PDS{{\mathit{PDS}}}

\newcommand{\rb}{\tikz[baseline=-2.75pt] \fill[color=red] (0,0) circle[radius=3pt]; \hspace{1.25pt}}
\newcommand{\bb}{\tikz[baseline=-2.75pt] \fill[color=blue] (0,0) circle[radius=2pt]; \hspace{1.25pt}}

\newcommand{\red}[1]{\textcolor{red}{#1}}
\newcommand{\blue}[1]{\textcolor{blue}{#1}}

\tikzset{every picture/.style={
baseline=0pt, line width=1pt, x=4mm, y=4mm}
}

\newcommand{\pc}[3]{
\draw[#1] (#3) --++(0,0.5) --++(1,1) -- ++(0,0.5);
\draw[#2] ($(#3)+(1,0)$) --++(0,0.5) --++(-0.3,0.3);
\draw[#2] ($(#3)+(0,2)$) --++(0,-0.5) --++(0.3,-0.3);
}
\newcommand{\nc}[3]{
\draw[#1] (#3) --++(0,0.5) --++(0.3,0.3);
\draw[#1] ($(#3)+(1,2)$) --++(0,-0.5) --++(-0.3,-0.3);
\draw[#2] ($(#3)+(1,0)$) --++(0,0.5) --++(-1,1) --+(0,0.5);
}

\newcommand{\npc}{
\begin{tikzpicture}[baseline=0.4pt, x=2.75mm, y=2.75mm]
\draw[->, semithick] (0,0) -- (1,1);
\draw[semithick] (1,0) -- (0.7, 0.3); 
\draw[->, semithick] (0.35, 0.65) -- (0,1);
\end{tikzpicture}
}

\newcommand{\nnc}{
\begin{tikzpicture}[baseline=0.4pt, x=2.75mm, y=2.75mm]
\draw[->, semithick] (1,0) -- (0,1);
\draw[semithick] (0,0) -- (0.3, 0.3); 
\draw[->, semithick] (0.65, 0.65) -- (1,1);
\end{tikzpicture}
}

\newcommand{\ndp}{
\begin{tikzpicture}[baseline=0.4pt, x=2.75mm, y=2.75mm]
\draw[->, semithick] (0,0) -- (1,1);
\draw[->, semithick] (1,0) -- (0,1);
\fill (0.5, 0.5) circle[radius=1pt];
\end{tikzpicture}
}

\tikzset{rs/.style={color=red, ultra thick}}
\tikzset{bs/.style={color=blue, semithick}}

\begin{document}

\title{Emergent version of Drinfeld's associator equations}

\author{Yusuke Kuno\thanks{Department of Mathematics, Tsuda University, 2-1-1 Tsuda-machi, Kodaira-shi, Tokyo 187-8577, Japan \texttt{e-mail:kunotti@tsuda.ac.jp}}}
\affil{with an appendix by Dror Bar-Natan}
\date{}

\maketitle

\begin{abstract}
The works of Alekseev and Torossian \cite{AT12} and Alekseev, Enriquez, and Torossian \cite{AET} show that any solution of Drinfeld's associator equations gives rise to a solution of the Kashiwara-Vergne equations in an explicit way.
We introduce a weak version of Drinfeld's associator equations that we call the {\it emergent} version of the original equations.
It is shown that solutions to the resulting linearized emergent Drinfeld's equations still lead to solutions to the linearized Kashiwara-Vergne equations. 

The emergent Drinfeld equations arise within a natural topological context of {\it emergent braids}, which we discuss.
Our results are adjacent to the results of Bar-Natan, Dancso, Hogan, Liu and Scherich \cite{LesDiablerets-2208, BNDHLS} on the relationship between emergent tangles and the Goldman-Turaev Lie bialgebra.
We hope that in time our results will play a role in relating several bodies of work, on Drinfeld associators, Kashiwara-Vergne equations, and on expansions for classical tangles, for w-tangles, and for the Goldman-Turaev Lie bialgebra.
\end{abstract}

\section{Introduction} \label{sec:Intro}

\subsection{Drinfeld associators and Kashiwara-Vergne equations} \label{subsec:AssocKV}

Alekseev and Torossian \cite{AT12} proved that any Drinfeld associator gives rise to a solution of the Kashiwara-Vergne (KV) problem \cite{KV78}. 
They reformulated the original KV problem in a universal form which involves the free Lie algebra in two variables.
The resulting (injective) map
\begin{equation} \label{eq:ASS_KV}
{\sf Assoc}_1 \hookrightarrow {\sf SolKV}
\end{equation}
from the set of Drinfeld associators (with coupling constant $1$) to the set of solutions to the KV problem has been made explicit by Alekseev, Enriquez and Torossian \cite{AET}.
There is also an explicit map between the corresponding spaces of infinitesimal deformations
\begin{equation} \label{eq:AT_nu}
\nu: \grt_1 \hookrightarrow \krv_2,
\end{equation}
where $\grt_1$ is the Grothendieck-Teichm\"{u}ller Lie algebra \cite[\S 5]{Drinfeld}, and $\krv_2$ is the Kashiwara-Vergne Lie algebra \cite[\S 4]{AT12}.
These graded Lie algebras integrate to the groups $\GRT_1$ and $\KRV_2$ which act freely and transitively on ${\sf Assoc}_1$ and ${\sf SolKV}$, respectively.

As was initially pointed out in \cite{Drinfeld}, the theory of Drinfeld associators has a topological nature.
Bar-Natan \cite{BN98} gave an interpretation of Drinfeld associators as $1$-formality isomorphisms (homomorphic expansions) of the category ${\bf PaB}$ of parenthesized braids.
The KV theory admits similar topological interpretations too, at least in two ways.
One is given by Bar-Natan and Dancso \cite{BND17} in terms of welded foams, a class of singular surfaces in $\mathbb{R}^4$.
The other is given by Alekseev, Kawazumi, Kuno and Naef \cite{AKKN18, AKKN_hg} in terms of the Goldman-Turaev Lie bialgebra \cite{Go86, Turaev91}, an algebraic structure on the free homotopy classes of loops in an oriented surface. 

With the topological interpretations of the Drinfeld associators and the KV theory mentioned above in mind, we introduce a weak version of the Drinfeld's associator equations that we call the {\it emergent} version of the original equations, and a graded vector space $\grt_1^{\EM}$ as the space of solutions to the linearized emergent equations.
Briefly, we obtain the emergent equations by working with certain subquotients of the target space of the original equations.
For instance, Drinfeld's pentagon equation for $\grt_1$ takes place in the Drinfeld-Kohno Lie algebra $\dk_4$ on four strands, which is generated by six elements $t_{ij}$ with $1 \le i < j \le 4$.
The corresponding emergent pentagon equation for $\grt_1^\EM$ takes place in $\edk_{2,2}$, which is defined to be the Lie subalgebra of $\dk_4$ generated by $t_{ij}$'s with $(i,j) \neq (1,2)$ modulo the commutator of the Lie ideal generated by $t_{34}$.
More generally, for any $m, n \ge 0$ we define a subquotient $\edk_{m,n}$ of the Drinfeld-Kohno Lie algebra $\dk_{m+n}$.
A topological context for $\edk_{m,n}$ is explained in Appendix.
We show that the map \eqref{eq:AT_nu} decomposes as 
\begin{equation} \label{eq:nu_decomp}
\nu: \grt_1 \hookrightarrow \grt_1^\EM \overset{\nu^\EM}{\hookrightarrow} \krv_2,
\end{equation}
and identify the image of $\nu^\EM$.

The emergent Drinfeld equations arise within a natural topological context of {\it emergent braids}.
As we will show, the defining equations for $\grt_1^\EM$ involve operations of the linearized version of the Goldman-Turaev Lie bialgebra.
We expect the emergent braids to serve as an intermediate object relating the topological aspect of Drinfeld associators and the KV theory.
In a future work, we hope to continue our study towards a decomposition of the map~\eqref{eq:ASS_KV} from the perspective of emergent braids.

\subsection{Statement of the main result} \label{sec:intro_main}

Let us give a precise statement about the decomposition \eqref{eq:nu_decomp}, which is the main result of this paper.
We need a few more details about the map~\eqref{eq:AT_nu}. 

Let $\lie_2 = \lie(x,y)$ be the free Lie algebra on two variables $x$ and $y$.
Recall that $\grt_1$ is the space of Lie polynomials $\psi = \psi(x,y) \in \lie_2$ satisfying a certain set of equations (one pentagon and two hexagon equations), and that the elements of $\krv_2$ are pairs $(u(x,y),v(x,y))$ of two Lie polynomials satisfying two equations (KV1) and (KV2) (see Section~\ref{subsec:KVliealg} for more precise definitions).
Then the map~\eqref{eq:AT_nu} is given by the following formula~\cite[Theorem~4.6]{AT12}:
\[
\nu(\psi) = (\psi(-x-y,x), \psi(-x-y, y)).
\]

Let $\ass_2 = \ass(x,y)$ be the free associative algebra on two variables $x$ and $y$.
One can naturally regard $\lie_2$ as a subspace of $\ass_2$.
The (non-commutative) partial differential operators $\pa_x, \pa_y : \lie_2 \to \ass_2$ are defined by the formula
$a = (\pa_x a) x + (\pa_y a) y$
for any $a \in \lie_2$.
Let $R : \lie_2 \to \ass_2$ be the unique linear map satisfying $R(x) = R(y) = 0$ and 
\begin{align*}
R([a,b]) =& [R(a), b] + [a, R(b)] \\ 
& + (\pa_x b) x \, \iota(\pa_x a) - (\pa_x a) x \, \iota(\pa_x b) +
 (\pa_y b) y \, \iota(\pa_y a) - (\pa_y a) y \, \iota(\pa_y b)
\end{align*}
for any $a, b\in \lie_2$.
Here, $\iota$ is the anti-automorphism of $\ass_2$ defined by $\iota(x) = -x$ and $\iota(y) = -y$; for instance, $\iota(xxy) = -yxx$.
Let $\grt_1^\EM$ be the space of Lie polynomials $\varphi = \varphi(x,y) \in \lie_2$ satisfying the following equations:
\begin{align}
& \varphi(y,0) - \varphi(x+y,0) = 0, \label{eq:intro_grtEM1} \\
& (\pa_y \varphi)(x,y) + (\pa_y \varphi)(y,0) - (\pa_y \varphi)(x+y,0) -R(\varphi) = 0, \label{eq:intro_grtEM2} \\
& [x, \varphi(y,x)] + [y, \varphi(x,y)] = 0. \label{eq:intro_grtEM3}
\end{align}
The equations \eqref{eq:intro_grtEM1} and \eqref{eq:intro_grtEM2} correspond to the emergent version of the pentagon equation.
The equation \eqref{eq:intro_grtEM3} appears by a rather technical reason, but it has a topological explanation as well.
By construction, there is an injection $\grt_1 \hookrightarrow \grt_1^\EM, \psi(x,y) \to \psi(-x-y, y)$.

An element $(u(x,y), v(x,y)) \in \krv_2$ is called symmetric if $v(x,y) = u(y,x)$. 
The space of symmetric elements in $\krv_2$ forms a Lie subalgebra denoted by $\krv_2^{\rm sym}$.
For a Lie polynomial $\varphi = \varphi(x,y) \in \lie_2$, set
\[
\nu^\EM(\varphi) := (\varphi(y,x), \varphi(x,y)).
\]
The main result of this paper is the following, which in particular proves the decomposition of $\nu$ shown in \eqref{eq:nu_decomp}.

\begin{thm} \label{thm:main}
\begin{enumerate}
\item[{\rm (i)}] For any $\varphi \in \grt_1^\EM$, we have $\nu^\EM(\varphi) \in \krv_2^{\rm sym}$.
\item[{\rm (ii)}] The map 
$
\nu^\EM: \grt_1^\EM \to (\krv_2^{\rm sym})_{\ge 2}
$
is a graded linear isomorphism.
\end{enumerate}
\end{thm}

It turns out that the space $\grt_1^\EM$ has a Lie algebra structure such that the injection $\grt_1 \to \grt_1^\EM$ and the map $\nu^\EM$ become Lie homomorphisms.
The Lie bracket $[\cdot, \cdot]^{\EM}$ on $\grt_1^\EM$ is given by the formula
\[
[\varphi_1, \varphi_2]^\EM = [\varphi_1, \varphi_2] + 
\rho_1(\varphi_2) - \rho_2(\varphi_1), 
\]
where the first term in the right hand side is the usual Lie bracket in $\lie_2$, and $\rho_i$ is the derivation on $\lie_2$ defined by $x \mapsto \varphi_i(y,x)$ and $y \mapsto \varphi_i(x,y)$.

\begin{remark}
It is not known whether $\krv_2^{\rm sym}$ coincides with $\krv_2$ or not \cite[Remark~8.10]{AT12}.
The map $\nu$ is conjectured to give a graded isomorphism $\grt_1 \cong (\krv_2)_{\ge 2}$ \cite[\S 4.2]{AT12}.
If this conjecture is correct, then it implies $\krv_2^{\rm sym} = \krv_2$ and $\grt_1 = \grt_1^{\EM}$.
\end{remark}

The map $R$ in the defining equation of $\grt_1^\EM$ is closely related to an operation of the linearization of the Goldman-Turaev Lie bialgebra of a punctured disk.
Our proof of Theorem~\ref{thm:main} is based on this fact and an interpretation of the KV theory in terms of the Goldman-Turaev Lie bialgebra \cite{AKKN18, AKKN_hg}.

We mention several recent works related to the content of this paper.
Our results and the idea of emergent knotted objects are adjacent to a work by Bar-Natan, Dancso, Hogan, Liu and Scherich \cite{LesDiablerets-2208, BNDHLS}, although there is a difference in the setting: the objects considered in \cite{LesDiablerets-2208, BNDHLS} are the quotient of our emergent objects by the Conway relation.
Works by Alekseev, Naef and Ren \cite{ANR24} and Naef and Navarro Betancourt \cite{NN_in_prep} discuss essentially the same object as ours but from different perspectives.

\subsection{Organization of the paper} \label{subsec:organization}

This paper has one appendix, written by Dror Bar-Natan, which explains the concept of ``emergent knotted objects''.
Reading it before proceeding to Section~\ref{sec:mbcd} will help the reader understand the philosophical background of this work.

In Sections~\ref{sec:mbcd} to \ref{sec:hexpmb}, we discuss emergent braids.
Actually, it is more natural to consider more general objects, and so we begin by introducing the concepts of mixed braids and chord diagrams in Section~\ref{sec:mbcd}.
In Section~\ref{sec:ebcd} we study the structure of the emergent Drinfeld-Kohno Lie algebra.
In Section~\ref{sec:hexpmb}, we introduce the parenthesized version of mixed braids and emergent braids, formulate the notion of $1$-formality isomorphisms (homomorphic expansions) of the categories of parenthesized mixed/emergent braids, and extract the emergent version of Drinfeld's associator equations.

The last two sections are devoted to the proof of Theorem~\ref{thm:main}.
In Section~\ref{sec:lopKV}, we recall necessary materials from \cite{AKKN18, AKKN_hg} on the (linearized) Goldman-Turaev Lie bialgebra of a punctured disk and its relation to the KV theory. 
In Section~\ref{sec:pfmain}, we prove Theorem~\ref{thm:main}.

\subsubsection*{Acknowledgements}
The content of this paper grew out of a collaboration with Dror Bar-Natan.
He coined the notion of ``emergent knotted objects'', and has spent considerable time with the author on the computer implementation of the emergent associator equations.
The author would like to express special thanks to him for collaborating on various stages of this work.
The author also would like to thank Hidekazu Furusho for valuable comments and suggestions.

This research is supported by JSPS KAKENHI 23K03121 and 24K00520.

\subsubsection*{Notation}

\begin{itemize}
\item 
Throughout this paper we work over the rationals $\Q$, though all of our arguments hold true over any field of characteristic zero.

\item 
For a nonnegative integer $n$, let $\ass_n$ be the free associative algebra on $n$ free generators.
When we need to specify generators, we write $\ass_n = \ass(x_1,\ldots,x_n)$ for example.

\item
The algebra $\ass_n$ has a structure of Hopf algebra whose coproduct, antipode and augmentation are given on generators by $\Delta(x_i) = x_i \otimes 1 + 1 \otimes x_i$, $\iota(x_i) = -x_i$ and $\varepsilon(x_i) = 1$.
We also use the notation $\overline{a} = \iota(a)$ for the antipode.

\item
We denote by $\lie_n = \lie(x_1,\ldots,x_n)$ the free Lie algebra on $n$ free generators $x_1,\ldots,x_n$.
One can identify $\lie_n$ with the space of primitive elements in $\ass_n$, namely $\lie_n = \{ a \in \ass_n \colon \Delta(a) = a \otimes 1 + 1 \otimes a\}$.
It holds that $\iota(a) = -a$ for any $a \in \lie_n$.

\item
Let $\mathcal{C}$ be a groupoid or more generally a category, and $O, O'$ objects in $\mathcal{C}$.
We denote by $\mathcal{C}(O, O')$ the set of morphisms in $\mathcal{C}$ from $O$ to $O'$.
\end{itemize}

\section{Mixed braids and chord diagrams} \label{sec:mbcd}

We introduce the notion of mixed braids.
Then we define the notion of mixed chord diagrams as the corresponding associated graded object.

\subsection{Mixed braids} \label{subsec:mb}

For a nonnegative integer $l$, let $B_l$ be Artin's braid group on $l$ strands.
Our convention about the product of $B_l$ is as follows: the product $\beta \beta'$ of two braids $\beta$ and $\beta'$ is the braid obtained by placing $\beta'$ above $\beta$.
For example,
\[
\begin{tikzpicture}[baseline=8pt]
\draw (0,0) -- (0,0.5) -- (1,1.5) --(1,2);
\draw (1,0) -- (1,0.5) -- (0.7, 0.8);
\draw (0.3, 1.2) -- (0,1.5) -- (0,2);
\draw (2,0) -- (2,2);
\end{tikzpicture}
\, \, \cdot \, \, 
\begin{tikzpicture}[baseline=8pt]
\draw (0,0) -- (0,2);
\draw (1,0) -- (1,0.5) -- (1.3, 0.8);
\draw (1.7,1.2) -- (2,1.5) -- (2,2);
\draw (2,0) -- (2,0.5) -- (1, 1.5);
\draw (1, 1.5) -- (1,2);
\end{tikzpicture}
=
\begin{tikzpicture}[baseline=14pt, x=4mm, y=3mm]
\draw (0,0) -- (0,0.5) -- (1,1.5) --(1,2);
\draw (1,0) -- (1,0.5) -- (0.7, 0.8);
\draw (0.3, 1.2) -- (0,1.5) -- (0,2);
\draw (2,0) -- (2,2);
\draw (0,2) -- (0,4);
\draw (1,2) -- (1,2.5) -- (1.3, 2.8);
\draw (1.7, 3.2) -- (2, 3.5) -- (2,4);
\draw (2,2) -- (2, 2.5) -- (1, 3.5);
\draw (1, 3.5) -- (1, 4);
\end{tikzpicture} \, .
\]

\begin{dfn}
Fix nonnegative integers $m$ and $n$.
A {\it mixed braid} of type $(m,n)$ is an element of $B_{m+n}$ equipped with a coloring of its strands with either red or blue such that 
\begin{itemize}
\item
there are $m$ red colored strands which we draw slightly thicker and $n$ blue colored strands which we draw slightly thinner, and 
\item 
if we forget all the blue colored strands and view the rest as an element in $B_m$, we are left with the trivial $m$-braid.
\end{itemize}
\end{dfn}

A blue colored strand in a mixed braid is simply called a {\it strand}, and a red colored strand is called a {\it pole}.

\begin{example} \label{exple:mb}
In the following three pictures, the first two pictures are mixed braids of type $(2,2)$.
Observe that their underlying braids on $2+2 = 4$ strands are the same.
However, the picture on the right is not a mixed braid.
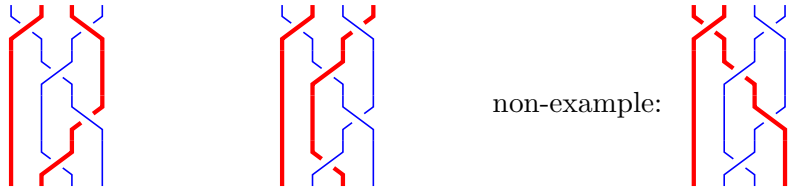
\[
\begin{tikzpicture}[baseline=28pt, x=4mm, y=3mm]
\draw[rs] (0,0) -- (0,6);
\pc{rs}{bs}{0,6}
\pc{rs}{bs}{1,0}
\nc{rs}{bs}{2,2}
\pc{bs}{bs}{1,4}
\nc{bs}{rs}{2,6}
\draw[bs] (3,0) -- (3,2);
\draw[bs] (1,2) -- (1,4);
\draw[rs] (3,4) -- (3,6);
\end{tikzpicture}
\hspace{6em}
\begin{tikzpicture}[baseline=28pt, x=4mm, y=3mm]
\draw[rs] (0,0) -- (0,6);
\pc{rs}{bs}{0,6}
\pc{bs}{rs}{1,0}
\nc{bs}{bs}{2,2}
\pc{rs}{bs}{1,4}
\nc{rs}{bs}{2,6}
\draw[bs] (3,0) -- (3,2);
\draw[rs] (1,2) -- (1,4);
\draw[bs] (3,4) -- (3,6);
\end{tikzpicture}
\hspace{4em}
\text{non-example:}
\quad 
\begin{tikzpicture}[baseline=28pt, x=4mm, y=3mm]
\draw[rs] (0,0) -- (0,6);
\pc{rs}{rs}{0,6}
\pc{bs}{bs}{1,0}
\nc{bs}{rs}{2,2}
\pc{bs}{rs}{1,4}
\nc{bs}{bs}{2,6}
\draw[rs] (3,0) -- (3,2);
\draw[bs] (1,2) -- (1,4);
\draw[bs] (3,4) -- (3,6);
\end{tikzpicture}
\]
\end{example}

We denote by $B_{m,n}$ the set of mixed braids of type $(m,n)$.
One can construct the product of two mixed braids $\beta, \beta'$ of the same type when the coloring of the strands of $\beta$ at the top matches that of $\beta'$ at the bottom.
In this manner, the set $B_{m,n}$ forms a groupoid.
Its set of objects is the set $W_{m,n}$ of words of length $m+n$ consisting of $m$ red (slightly bigger) bullets $\rb$ and $n$ blue (slightly smaller) bullets $\bb$. 
When $o \in W_{m,n}$, the word $o$ is called of type $(m,n)$.
For  $o, o' \in W_{m,n}$, we denote by $B_{m,n}(o,o')$ the set of mixed braids whose bottom and top ends match $o$ and $o'$, respectively.
For example, the leftmost picture in Example~\ref{exple:mb} is an element in $B_{2,2}(\rb \rb \bb \bb, \bb \rb \rb \bb)$.

\begin{dfn}
Let $m, n \ge 0$ and let $o, o' \in W_{m,n}$.
A {\it mixed permutation} (of type $(m,n)$) from $o$ to $o'$ is a permutation $\sigma$ of $m+n$ letters such that 
\begin{itemize}
\item for any $1 \le i \le m+n$, the $i$th letter of $o$ and the $\sigma(i)$th letter of $o'$ have the same color, and 
\item if we forget all the blue letters in $o$ and $o'$ and view the restriction of $\sigma$ to the red bullets as a permutation of $m$ letters, then it is trivial.
\end{itemize}
\end{dfn}

Alternatively, a mixed permutation is a mixed braid whose over/under information at each crossing of strands is lost.
For example, 
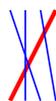
\[
\begin{tikzpicture}[x=2mm, y=3mm]
\draw[rs] (0,0) -- (3,4);
\draw[bs] (1,0) -- (1,4);
\draw[bs] (2,0) -- (0,4);
\draw[bs] (3,0) -- (2,4);
\end{tikzpicture}
\]
is a mixed permutation from $\rb \bb \bb \bb$ to $\bb \bb \bb \rb$ given by $1 \mapsto 4, 2 \mapsto 2, 3 \mapsto 1$, and $4 \mapsto 3$.

For $o, o' \in W_{m,n}$, we denote by $\mathfrak{S}_{m,n}(o,o')$ the set of mixed permutations from $o$ to $o'$.
The set $\mathfrak{S}_{m,n} = \bigsqcup_{o, o' \in W_{m,n}} \mathfrak{S}_{m,n}(o,o')$ of all mixed permutations of type $(m,n)$ naturally forms a groupoid.
The forgetful map 
\[
\pi : B_{m,n} \to \mathfrak{S}_{m,n}
\]
is a homomorphism of groupoids.

Let
$
o^\std_{m,n} := \underbrace{\rb \red{\cdots} \rb}_m \underbrace{\bb \blue{\cdots} \bb}_n \in W_{m,n}
$, 
then the set $B^\std_{m,n} := B_{m,n}(o^\std_{m,n}, o^\std_{m,n})$ forms a group with respect to the groupoid structure of $B_{m,n}$.
One can regard $B_{m,n}^\std$ as a subgroup of $B_{m+n}$ in a natural way.

The trivial permutation of degree $m+n$ defines the mixed permutation
\[
1_{m,n}:= 
\begin{tikzpicture}[baseline=12pt, x=3mm, y=3mm]
\draw[rs] (0,0) -- (0,4);
\draw[rs] (3,0) -- (3,4);
\draw[bs] (4,0) -- (4,4);
\draw[bs] (8,0) -- (8,4);
\node at (1.75,0.5) {$\cdots$};
\node at (6,0.5) {$\cdots$};
\node at (0,-0.5) {\small $1$};
\node at (3,-0.5) {\small $m$};
\node at (4,-0.5) {\small $1$};
\node at (8,-0.5) {\small $n$};
\end{tikzpicture}
\quad \in \mathfrak{S}_{m,n}(o^\std_{m,n}, o^\std_{m,n}).
\]
Then, $P_{m,n} : = \pi^{-1}(1_{m,n})$ is a normal subgroup of $B^\std_{m,n}$.
We call $P_{m,n}$ the {\it mixed pure braid group} of type $(m,n)$.
In fact, Lambropoulou~\cite[Sections 2 and 3]{Lam00} introduced the same group with the same notation and gave its explicit presentation.
In particular, $P_{m,n}$ is generated by the following elements $\alpha_{ij}$, where $1\le i \le m$, $1 \le j\le n$, and $\tau_{ij}$, where $1 \le i<j \le n$:
\[
\alpha_{ij} = 
\begin{tikzpicture}[baseline=18pt]
\draw[rs] (0,0) -- (0,4);
\draw[rs] (2,0) -- (2,2.3);
\draw[rs] (2,2.6) -- (2,4);
\draw[rs] (4,0) -- (4,1.1);
\draw[rs] (4,1.4) -- (4,2.65);
\draw[rs] (4,2.95) -- (4,4);
\draw[bs] (5,0) -- (5,0.9);
\draw[bs] (5,1.2) -- (5,2.85);
\draw[bs] (5,3.15) -- (5,4);
\draw[bs] (7,0) -- (7,0.55);
\draw[bs] (7,0.85) -- (7,3.15);
\draw[bs] (7,3.45) -- (7,4);
\draw[bs] (8,0) -- (8,0.5) -- (2.2,1.55);
\draw[bs] (1.8,2.4) -- (8,3.5) -- (8,4);
\draw[bs] (10,0) -- (10,4);
\draw[bs] (1.8,1.62) to[bend left=90] (1.8,2.4);
\node at (1,0.5) {$\cdots$};
\node at (3,0.5) {$\cdots$};
\node at (6,0.5) {$\cdots$};
\node at (9,0.5) {$\cdots$};
\node at (0,-0.5) {\small $1$};
\node at (2,-0.5) {\small $i$};
\node at (4,-0.5) {\small $m$};
\node at (5,-0.5) {\small $1$};
\node at (8,-0.5) {\small $j$};
\node at (10,-0.5) {\small $n$};
\end{tikzpicture},
\qquad
\tau_{ij} = 
\begin{tikzpicture}[baseline=18pt]
\draw[rs] (0,0) --(0,4);
\draw[rs] (2,0) --(2,4);
\draw[bs] (3,0) --(3,4);
\draw[bs] (5,0) --(5,2.3);
\draw[bs] (5,2.6) --(5,4);
\draw[bs] (6,0) --(6,1.2);
\draw[bs] (6,1.5) --(6,2.6);
\draw[bs] (6,2.9) --(6,4);
\draw[bs] (8,0) --(8,0.6);
\draw[bs] (8,0.9) --(8,3.1);
\draw[bs] (8,3.4) --(8,4);
\draw[bs] (9,0) --(9,0.5) --(5.2, 1.55);
\draw[bs] (9,4) --(9,3.5) --(4.8, 2.4);
\draw[bs] (4.8, 1.62) to[bend left=90] (4.8, 2.4);
\draw[bs] (11,0) --(11,4);
\node at (1,0.5) {$\cdots$};
\node at (4,0.5) {$\cdots$};
\node at (7,0.5) {$\cdots$};
\node at (10,0.5) {$\cdots$};
\node at (0,-0.5) {\small $1$};
\node at (2,-0.5) {\small $m$};
\node at (3,-0.5) {\small $1$};
\node at (5,-0.5) {\small $i$};
\node at (9,-0.5) {\small $j$};
\node at (11,-0.5) {\small $n$};
\end{tikzpicture}.
\]

Collecting all types of mixed braids and mixed permutations we consider the groupoids $B_{\rb \bb}:= \bigsqcup_{m,n \ge 0} B_{m,n}$ and $\mathfrak{S}_{\rb \bb} := \bigsqcup_{m,n \ge 0} \mathfrak{S}_{m,n}$.
Both of them have $W_{\rb \bb} := \bigsqcup_{m,n \ge 0} W_{m,n}$ as the set of objects.
Following the treatment in \cite[Section 2.2.1]{BN98}, we define the category ${\bf MB}$ of mixed braids as a $\Q$-linear extension of the groupoid $B_{\rb \bb}$.
Its set of objects is $W_{\rb \bb}$.
Let $o, o' \in W_{\rb \bb}$.
If the types of $o$ and $o'$ are different, there is no morphism from $o$ to $o'$.
If not, then morphisms from $o$ to $o'$ are pairs $( \sum_j c_j \beta_j, \sigma)$, where $\sigma \in \mathfrak{S}_{\rb \bb}(o, o')$ and $\sum_j c_j \beta_j$ is a $\Q$-linear combination of mixed braids such that $\pi(\beta_j) = \sigma$ for all $j$.
Thus when the types of $o$ and $o'$ are the same, we have ${\bf MB}(o,o') = \bigsqcup_{\sigma \in \mathfrak{S}_{\rb \bb}(o,o')} {\bf MB}(o,o')_{\sigma}$, where the subscript $\sigma$ stands for consisting of elements which have $\sigma$ as the second entry.
The composition in ${\bf MB}$ is naturally induced from the composition in $B_{\rb \bb}$.

\subsection{Mixed version of the Drinfeld-Kohno Lie algebra} \label{subsec:mdk}

Let $n$ be a nonnegative integer.
Recall that the {\it Drinfeld-Kohno Lie algebra}, which we denote by $\dk_n$, is the graded Lie algebra generated by
degree one elements $t_{ij} = t_{ji}$ for $1\le i \neq j \le n$ subject to the commutation relation $[t_{ij},t_{kl}] = 0$ for distinct indices $i,j,k,l$, and the 4T relation $[t_{ij} + t_{jk}, t_{ik}] = 0$ for distinct indices $i, j, k$.
In a diagrammatic language, $\dk_n$ is the  Lie algebra of horizontal chord diagrams on $n$ vertical lines, and the generator $t_{ij}$ corresponds to the chord diagram consisting of a single chord connecting the $i$th and $j$th lines:
\[
t_{ij} = 
\begin{tikzpicture}[baseline=18pt]
\draw (0,0) -- (0,4);
\draw (2,0) -- (2,4);
\draw (4,0) -- (4,4);
\draw (6,0) -- (6,4);
\draw (2,2) -- (4,2);
\node at (1,0.5) {$\cdots$};
\node at (3,0.5) {$\cdots$};
\node at (5,0.5) {$\cdots$};
\node at (0,-0.5) {\small $1$};
\node at (2,-0.5) {\small $i$};
\node at (4,-0.5) {\small $j$};
\node at (6,-0.5) {\small $n$};
\end{tikzpicture}.
\]
For every $n >0$, there is a semi-direct product decomposition
\begin{equation} \label{eq:dk_decomp}
\dk_n = \dk_{n-1} \ltimes \lie(t_{1n},\ldots, t_{(n-1)n}).
\end{equation}

It is known that the universal enveloping algebra of $\dk_n$ is isomorphic to the associated graded of the group algebra of the $n$-strand pure braid group with respect to the powers of the augmentation ideal \cite{Kohno}.
See also \cite[Theorem 10.0.4]{FresseI}.
With this in mind, we introduce a variant of $\dk_n$ corresponding to the group $P_{m,n}$.

\begin{dfn}
For $m,n \ge 0$, let $\dk_{m,n}$ be the graded Lie algebra generated by degree one elements $a_{ij}$ for $1 \le i \le m$, $1\le j \le n$ and $c_{ij} = c_{ji}$ for $1 \le i \neq j \le n$, subject to the commutation and 4T relations among them, where we regard $a_{ij}= t_{i(m+j)}$ and $c_{ij} = t_{(m+i)(m+j)}$ as the corresponding generators of $\dk_{m+n}$.
\end{dfn}

Diagrammatically, the generators of $\dk_{m,n}$ are horizontal chord diagrams with a single chord on $m$ vertical red lines and $n$ vertical blue lines:
\[
a_{ij} = 
\begin{tikzpicture}[baseline=18pt]
\draw[rs] (0,0) -- (0,4);
\draw[rs] (2,0) -- (2,4);
\draw[rs] (4,0) -- (4,4);
\draw (2,2) -- (7,2);
\draw[bs] (5,0) -- (5,4);
\draw[bs] (7,0) -- (7,4);
\draw[bs] (9,0) -- (9,4);
\node at (1,0.5) {$\cdots$};
\node at (3,0.5) {$\cdots$};
\node at (6,0.5) {$\cdots$};
\node at (8,0.5) {$\cdots$};
\node at (0,-0.5) {\small $1$};
\node at (2,-0.5) {\small $i$};
\node at (4,-0.5) {\small $m$};
\node at (5,-0.5) {\small $1$};
\node at (7,-0.5) {\small $j$};
\node at (9,-0.5) {\small $n$};
\end{tikzpicture},
\qquad
c_{ij} = 
\begin{tikzpicture}[baseline=18pt]
\draw[rs] (0,0) --(0,4);
\draw[rs] (2,0) --(2,4);
\draw[bs] (3,0) --(3,4);
\draw[bs] (5,0) --(5,4);
\draw[bs] (7,0) --(7,4);
\draw[bs] (9,0) --(9,4);
\draw (5,2) -- (7,2); 
\node at (1,0.5) {$\cdots$};
\node at (4,0.5) {$\cdots$};
\node at (6,0.5) {$\cdots$};
\node at (8,0.5) {$\cdots$};
\node at (0,-0.5) {\small $1$};
\node at (2,-0.5) {\small $m$};
\node at (3,-0.5) {\small $1$};
\node at (5,-0.5) {\small $i$};
\node at (7,-0.5) {\small $j$};
\node at (9,-0.5) {\small $n$};
\end{tikzpicture}.
\]

\begin{remark}
We have $\dk_{0,n} = \dk_n$.
\end{remark}

The semi-direct product decomposition \eqref{eq:dk_decomp} generalizes to $\dk_{m,n}$:

\begin{lem} \label{lem:dk_decomp}
There is a semi-direct product decomposition of Lie algebra
\[
\dk_{m,n} = \dk_{m,n-1} \ltimes
\lie(a_{1n},\ldots,a_{mn},c_{1n},\ldots,c_{(n-1)n}).
\]
\end{lem}

\begin{proof}
We simply write $\lie(a,c) = \lie(a_{1n},\ldots,a_{mn},c_{1n},\ldots,c_{(n-1)n})$.
First we describe the Lie action $\rho$ of $\dk_{m,n-1}$ on $\lie(a,c)$ that is used in forming the semi-direct product $\dk_{m,n-1} \ltimes \lie(a,c)$.
It is specified by the value on generators of $\dk_{m,n-1}$:
for $1 \le i \le m$, $1\le j \le n-1$, $1\le k \le m$ and $1 \le l \le n-1$, 
\[
\rho(a_{ij}) (a_{kn}) = \begin{cases}
0 & (i \neq k) \\
-[c_{jn}, a_{kn}] & (i=k)
\end{cases},
\quad
\rho(a_{ij})(c_{ln}) = \begin{cases}
0 & (j \neq l) \\
-[a_{in}, c_{ln}] & (j = l)
\end{cases},
\]
and for $1 \le i \neq j \le n-1$, $1 \le k \le m$ and $1 \le l \le n-1$, 
\[
\rho(c_{ij})(a_{kn}) = 0,
\qquad
\rho(c_{ij})(c_{ln}) = 
\begin{cases}
0 & ( l \notin \{ i,j \}) \\
-[c_{jn}, c_{ln}] & (i = l)
\end{cases}.
\]
Note that these formulas are compatible with the Lie bracket in $\dk_{m,n}$. 
For example, we have $[a_{kj}, a_{kn}] = - [c_{jn}, a_{kn}]$ by the 4T relation, and this matches the value $\rho(a_{kj})(a_{kn}) = -[c_{jn},a_{kn}]$.
Now we define the map
$\dk_{m,n} \to \dk_{m,n-1} \ltimes \lie(a,c)$ by 
\[
a_{ij} \mapsto \begin{cases}
(a_{ij}, 0) & (j\le n-1) \\
(0, a_{in}) & (j = n)
\end{cases},
\qquad
c_{ij} \mapsto \begin{cases}
(c_{ij},0) & (j \le n-1) \\
(0, c_{in}) & (j = n)
\end{cases}.
\]
Then one can check that this map is a Lie algebra isomorphism.
\end{proof}

\begin{remark} \label{rem:dkmninj}
By Lemma~\ref{lem:dk_decomp}, we inductively see that the map $\dk_{m,n} \rightarrow \dk_{m+n}$ defined by $a_{ij} \mapsto t_{i(m+j)}$ and $c_{ij} \mapsto t_{(m+i)(m+j)}$ is an injective Lie homomorphism.
Therefore, one can identify $\dk_{m,n}$ with the Lie subalgebra of $\dk_{m+n}$ generated by $t_{i(m+j)}$ ($1\le i \le m$, $1\le j \le n$) and $t_{(m+i)(m+j)}$ ($1 \le i < j \le n$).
\end{remark}

We show how the Lie algebra $\dk_{m,n}$ and the group $P_{m,n}$ are related.
On the one hand, let $\mathcal{A}_{m,n} = U(\dk_{m,n})$ be the universal enveloping algebra of $\dk_{m,n}$.
It is an associative $\Q$-algebra generated by the same generators $a_{ij}$ and $c_{ij}$ as those of $\dk_{m,n}$, subject to the same relations as those of $\dk_{m,n}$, where we regard bracket symbol as commutator: $[a,b] = ab - ba$. 
On the other hand, the powers of the augmentation ideal $I = I P_{m,n}$ define a decreasing filtration of $\Q P_{m,n}$.
Thus one can construct the associated graded $\mathcal{A}_{\Q P_{m,n}}:= \prod_{k \ge 0} I^k/ I^{k+1}$ of the filtered algebra $\Q P_{m,n}$.

\begin{prop} \label{prop:grpureemb}
There is a canonical isomorphism of graded $\Q$-algebras $\mathcal{A}_{\Q P_{m,n}} \cong \mathcal{A}_{m,n}$, through which the class of $\alpha_{ij} - 1$ corresponds to $a_{ij}$ and the class of $\tau_{ij} - 1$ to $c_{ij}$.
\end{prop}

\begin{proof}
The proof is similar to the proof of the isomorphism $\mathcal{A}_{\Q P_n} \cong U(\dk_n)$ given in \cite[Theorem 10.0.4]{FresseI}, so we just give a sketch.
We start with the fact that there is a semi-direct product decomposition
\[
P_{m,n} \cong P_{m,n-1} \ltimes F_{m+n-1},
\]
where $F_{m+n-1}$ is the free group generated by $\alpha_{in}$, $1\le i \le m$ and $\tau_{in}$, $1 \le i \le n-1$
(see \cite[Section 3]{Lam00}).
Here, the action of $P_{m,n-1}$ on $F_{m+n-1}$ is by conjugation and hence is trivial on the abelianization of $F_{m+n-1}$.
Applying \cite[Proposition 8.5.7]{FresseI}, one has $\mathcal{A}_{\Q P_{m,n}} \cong 
\left( \mathcal{A}_{\Q P_{m,n-1}} \right) \sharp \left( \mathcal{A}_{\Q F_{m+n-1}} \right)$, where $\sharp$ denotes the semi-direct product of Hopf algebras.
Note that $\mathcal{A}_{\Q F_{m+n-1}}$ is naturally isomorphic to $\ass_{m+n-1}$, and Lemma~\ref{lem:dk_decomp} implies that there is an isomorphism $\mathcal{A}_{m,n-1} \sharp \ass_{m+n-1} \cong \mathcal{A}_{m,n}$.
Hence we can prove $\mathcal{A}_{\Q P_{m,n}} \cong \mathcal{A}_{m,n}$ by induction on $n$.
One can check that this isomorphism maps the class of $\alpha_{ij} -1$ to $a_{ij}$ and the class of $\tau_{ij} - 1$ to $c_{ij}$.
\end{proof}

\subsection{Operadic structure and coface maps} \label{subsec:cfc_mps}

There are naturally defined operations on mixed braids.
Let $\beta \in B_{m,n}$ be a mixed braid.
\begin{itemize}
\item
Extension operations.
We denote by $\delta_0^p (\beta)$ (resp. $\delta_0^s (\beta)$) be the mixed braid of type $(m+1,n)$ (resp. of type $(m, n+1)$) obtained from $\beta$ by adding a red (resp. blue) straight strand on the left.
Similarly, we define $\delta_{m+n+1}^p (\beta)$ (resp. $\delta_{m+n+1}^s (\beta)$) by adding a red strand (resp. blue strand) on the right.
For example, 
\[
\delta_0^p \left( \, 
\begin{tikzpicture}[baseline=18pt, x=2.4mm, y=1.8mm]
\draw[rs] (0,0) -- (0,6);
\pc{rs}{bs}{0,6}
\pc{rs}{bs}{1,0}
\nc{rs}{bs}{2,2}
\pc{bs}{bs}{1,4}
\nc{bs}{rs}{2,6}
\draw[bs] (3,0) -- (3,2);
\draw[bs] (1,2) -- (1,4);
\draw[rs] (3,4) -- (3,6);
\end{tikzpicture}
\, \right)
= \begin{tikzpicture}[baseline=18pt, x=2.4mm, y=1.8mm]
\draw[rs] (-1,0) -- (-1,8);
\draw[rs] (0,0) -- (0,6);
\pc{rs}{bs}{0,6}
\pc{rs}{bs}{1,0}
\nc{rs}{bs}{2,2}
\pc{bs}{bs}{1,4}
\nc{bs}{rs}{2,6}
\draw[bs] (3,0) -- (3,2);
\draw[bs] (1,2) -- (1,4);
\draw[rs] (3,4) -- (3,6);
\end{tikzpicture}
\qquad \text{and} \qquad 
\delta_5^s \left( \, 
\begin{tikzpicture}[baseline=18pt, x=2.4mm, y=1.8mm]
\draw[rs] (0,0) -- (0,6);
\pc{rs}{bs}{0,6}
\pc{rs}{bs}{1,0}
\nc{rs}{bs}{2,2}
\pc{bs}{bs}{1,4}
\nc{bs}{rs}{2,6}
\draw[bs] (3,0) -- (3,2);
\draw[bs] (1,2) -- (1,4);
\draw[rs] (3,4) -- (3,6);
\end{tikzpicture}
\, \right)
= \begin{tikzpicture}[baseline=18pt, x=2.4mm, y=1.8mm]
\draw[rs] (0,0) -- (0,6);
\pc{rs}{bs}{0,6}
\pc{rs}{bs}{1,0}
\nc{rs}{bs}{2,2}
\pc{bs}{bs}{1,4}
\nc{bs}{rs}{2,6}
\draw[bs] (3,0) -- (3,2);
\draw[bs] (1,2) -- (1,4);
\draw[rs] (3,4) -- (3,6);
\draw[bs] (4,0) -- (4,8);
\end{tikzpicture}
\, .
\]

\item
Cabling operations. 
For $1 \le i \le m+n$, let $\delta_i (\beta)$ be the mixed braid obtained from $\beta$ by doubling its $i$th strand, where we count strands at the bottom end of $\beta$.
The two strands newly created inherits the color of the original strand.
For example, 
\[
\delta_1\left( \, 
\begin{tikzpicture}[baseline=18pt, x=2.4mm, y=1.8mm]
\draw[rs] (0,0) -- (0,6);
\pc{rs}{bs}{0,6}
\pc{rs}{bs}{1,0}
\nc{rs}{bs}{2,2}
\pc{bs}{bs}{1,4}
\nc{bs}{rs}{2,6}
\draw[bs] (3,0) -- (3,2);
\draw[bs] (1,2) -- (1,4);
\draw[rs] (3,4) -- (3,6);
\end{tikzpicture}
\, \right)
= 
\begin{tikzpicture}[baseline=22pt, x=2.4mm, y=1.8mm]
\draw[rs] (-1,0) -- (-1,8);
\draw[rs] (1,8) -- (1,10);
\draw[rs] (2,8) -- (2,10);
\draw[bs] (3,8) -- (3,10);
\pc{rs}{bs}{-1,8};
\draw[rs] (0,0) -- (0,6);
\pc{rs}{bs}{0,6}
\pc{rs}{bs}{1,0}
\nc{rs}{bs}{2,2}
\pc{bs}{bs}{1,4}
\nc{bs}{rs}{2,6}
\draw[bs] (3,0) -- (3,2);
\draw[bs] (1,2) -- (1,4);
\draw[rs] (3,4) -- (3,6);
\end{tikzpicture}
\qquad \text{and} \qquad 
\delta_4 \left( \, 
\begin{tikzpicture}[baseline=18pt, x=2.4mm, y=1.8mm]
\draw[rs] (0,0) -- (0,6);
\pc{rs}{bs}{0,6}
\pc{rs}{bs}{1,0}
\nc{rs}{bs}{2,2}
\pc{bs}{bs}{1,4}
\nc{bs}{rs}{2,6}
\draw[bs] (3,0) -- (3,2);
\draw[bs] (1,2) -- (1,4);
\draw[rs] (3,4) -- (3,6);
\end{tikzpicture}
\, \right)
= \begin{tikzpicture}[baseline=26pt, x=2.4mm, y=1.8mm]
\draw[rs] (0,0) -- (0,8);
\pc{rs}{bs}{0,8}
\pc{rs}{bs}{1,0}
\nc{rs}{bs}{2,2}
\pc{bs}{bs}{1,4}
\nc{rs}{bs}{3,4}
\pc{bs}{bs}{2,6}
\pc{rs}{bs}{1,10}
\nc{bs}{rs}{3,10}
\draw[bs] (3,0) -- (3,2);
\draw[bs] (4,0) -- (4,2);
\draw[bs] (1,2) -- (1,4);
\draw[bs] (4,2) -- (4,4);
\draw[bs] (1,6) -- (1,8);
\draw[rs] (4,6) -- (4,8);
\draw[bs] (2,8) -- (2,10);
\draw[bs] (3,8) -- (3,10);
\draw[rs] (4,8) -- (4,10);
\draw[bs] (0,10) -- (0,12);
\end{tikzpicture}.
\]

\item
Changing a pole to a strand.
For $1 \le i \le m$, let $\vartheta_i(\beta)$ be the mixed braid obtained from $\beta$ by changing the $i$th red strand to a blue strand. 
For example, 
\[
\vartheta_1 \left( \, 
\begin{tikzpicture}[baseline=18pt, x=2.4mm, y=1.8mm]
\draw[rs] (0,0) -- (0,6);
\pc{rs}{bs}{0,6}
\pc{rs}{bs}{1,0}
\nc{rs}{bs}{2,2}
\pc{bs}{bs}{1,4}
\nc{bs}{rs}{2,6}
\draw[bs] (3,0) -- (3,2);
\draw[bs] (1,2) -- (1,4);
\draw[rs] (3,4) -- (3,6);
\end{tikzpicture}
\, \right)
= 
\begin{tikzpicture}[baseline=18pt, x=2.4mm, y=1.8mm]
\draw[bs] (0,0) -- (0,6);
\pc{bs}{bs}{0,6}
\pc{rs}{bs}{1,0}
\nc{rs}{bs}{2,2}
\pc{bs}{bs}{1,4}
\nc{bs}{rs}{2,6}
\draw[bs] (3,0) -- (3,2);
\draw[bs] (1,2) -- (1,4);
\draw[rs] (3,4) -- (3,6);
\end{tikzpicture}
\, .
\]
\end{itemize}

The operations defined above have counterparts in $\dk_{m,n}$.

\begin{itemize}
\item 
Extension operations.
Let $\delta_0 = \delta_0^p : \dk_{m,n} \to \dk_{m+1,n}$ (resp. $\delta_{m+n+1} = \delta_{m+n+1}^s : \dk_{m,n} \to \dk_{m,n+1}$) be the Lie homomorphism defined by 
$a_{ij} \mapsto a_{(i+1)j}$ and $c_{ij} \mapsto c_{ij}$ (resp. $a_{ij} \mapsto a_{ij}$ and $c_{ij} \mapsto c_{ij}$).

\item 
Cabling operations.
For $1 \le k \le m$, we define the Lie homomorphism $\delta_k : \dk_{m,n} \to \dk_{m+1,n}$ by 
\[
\delta_k(a_{ij}) = \begin{cases}
a_{ij} & ( 1\le i \le k-1) \\
a_{kj} + a_{(k+1)j} & ( i = k) \\
a_{(i+1)j} & ( k+1 \le i \le m)
\end{cases},
\qquad
\delta_k(c_{ij}) = c_{ij}.
\]
For $1 \le k \le n$, we define $\delta_{m + k} : \dk_{m,n} \to \dk_{m,n+1}$ by
\[
\delta_{m+k}(a_{ij}) = \begin{cases}
a_{ij} & ( 1\le j \le k-1) \\
a_{ik} + a_{i(k+1)} & ( j = k) \\
a_{i(j+1)} & (k+1 \le j \le n) 
\end{cases}
\]
and 
\[
\delta_{m+k}(c_{ij}) = \begin{cases}
c_{ij} & (j < k) \\
c_{ik} + c_{i(k+1)} & ( j=k) \\
c_{i(j+1)} & ( i < k < j) \\
c_{k(j+1)} + c_{(k+1)(j+1)} & (i = k) \\
c_{(i+1)(j+1)} & ( k < i)
\end{cases}.
\]

\item 
Changing a pole to a strand.
For the sake of simplicity we only introduce this operation applied to the last pole.
Let $\vartheta_m : \dk_{m,n} \to \dk_{m-1, n+1}$ be the Lie homomorphism defined by 
\[
\vartheta_m(a_{ij}) = \begin{cases}
a_{i(j+1)} & ( i < m) \\
c_{1(j+1)} & ( i = m)
\end{cases},
\qquad
\vartheta_m(c_{ij}) = c_{(i+1)(j+1)}.
\]
\end{itemize}

Using these operations, we define coface maps and a differential on $\dk_{m,n}$.

\begin{dfn}
For $0 \le k \le m+n+1$, we define the map $d_k = d_k^{m,n} : \dk_{m,n} \to \dk_{m,n+1}$ as follows:
\[
d_k = \begin{cases}
\vartheta_{m+1} \circ \delta_k & (0 \le k \le m) \\ 
\delta_k & (m+1 \le k \le m+n+1)
\end{cases}.
\] 
Furthermore, we set $d^{m,n} := \sum_{k=0}^{m+n+1} (-1)^k d_k : \dk_{m,n} \to \dk_{m,n+1}$.
\end{dfn}

The family of maps $\{d^{m,n} \}_n$ is indeed a differential.

\begin{lem}
We have $d^{m,n+1} \circ d^{m,n} = 0 : \dk_{m,n} \to \dk_{m,n+2}$.
\end{lem}

\begin{proof}
The proof is straightforward by using the relation $d_i \circ d_j = d_{j+1} \circ d_i$ for $i \le j$, which can be checked directly.
\end{proof}

\section{Emergent braids and chord diagrams} \label{sec:ebcd}

In this section, we introduce the notion of emergent braids and chord diagrams.
In particular, we describe the structure of the emergent version of the Drinfeld-Kohno Lie algebra.

\subsection{Emergent braids} \label{subsec:eb}

The group $B^\std_{m,n}$ acts on its normal subgroup $P_{m,n}$ by conjugation, and this extends linearly to an action on the group algebra $\Q P_{m,n}$.
We denote by $J$ the two-sided ideal of $\Q P_{m,n}$ generated by $\tau_{ij} - 1$, $1 \le i<j \le n$.
The powers $J^l$, $l \ge 0$, define a $B^\std_{m,n}$-invariant decreasing filtration of $\Q P_{m,n}$.

\begin{dfn} \label{dfn:emb}
For each $k \ge 1$ we set $\Q P_{m,n}^{/k}:= \Q P_{m,n} / J^k$.
In particular, the {\it algebra of emergent pure braids} of type $(m,n)$ is defined to be
\[
\Q P_{m,n}^\EM:= \Q P_{m,n}^{/2} = \Q P_{m,n}/J^2.
\]
\end{dfn}

\begin{remark} \label{rem:why}
Why ``emergent''?
In primary school language, ``Dror has an emergent knowledge of the French language'' means ``Dror knows French just a bit better than nothing at all''.
In a similar way, $\Q P_{m,n}^{/1}$ means ``no braiding phenomenon yet'', for in $\Q P_{m,n}^{/1}$ the blue strands are fully transparent to each other, and $\Q P_{m,n}^{/2}$ is ``emergent braiding'', for after moding out by $J^2$ just a whiff of braiding remains.
\end{remark}

The ideal $J$ of $\Q P_{m,n} = {\bf MB}(o^\std_{m,n}, o^\std_{m,n})_{1_{m,n}}$ and its powers extend to a multiplicative filtration of the category ${\bf MB}$ in the following way.
Let $o, o' \in W_{m,n}$ for some $m, n \ge 0$ and let $\sigma \in \mathfrak{S}_{\rb \bb}(o,o')$.
One can take mixed braids $\beta \in B_{m,n}(o^\std_{m,n}, o)$ and $\beta' \in B_{m,n}(o^\std_{m,n}, o')$ such that $\sigma = \pi(\beta)^{-1} \pi(\beta')$.
Then, the map 
$
\Q P_{m,n} \to {\bf MB}(o,o')_{\sigma}, 
u \mapsto \beta^{-1} u \beta'
$
is a $\Q$-linear isomorphism.
Since the ideal $J$ is $B^\std_{m,n}$-invariant, it follows that the subspaces $J^l_\sigma := \beta^{-1} J^l \beta'$, $l\ge 0$, are independent of the choice of $\beta$ and $\beta'$.
The collection $\{ J^l_{\sigma} \}_{l \ge 0, \sigma \in \mathfrak{S}_{\rb \bb}}$ is multiplicative in the sense that $J^l_{\sigma} \cdot J^{l'}_{\sigma'} \subset J^{l+l'}_{\sigma \sigma'}$ holds for any $l, l' \ge 0$ whenever $\sigma$ and $\sigma'$ are composable.

For each $k\ge 1$, we define the category ${\bf MB}^{/k}$ as follows.
The set of objects is $W_{\rb \bb}$.
For $o, o' \in W_{\rb \bb}$, the set of morphisms from $o$ to $o'$ is
\[
{\bf MB}^{/k}(o, o') := \begin{cases}
\bigsqcup_{\sigma \in \mathfrak{S}_{\rb \bb}(o,o')} \displaystyle\frac{{\bf MB}(o,o')_{\sigma}}{J^k_{\sigma}} & \text{if $o$ and $o'$ have the same type}, \\
\emptyset & \text{otherwise}.
\end{cases}
\]
The composition in ${\bf MB}^{/k}$ is induced from the composition in ${\bf MB}$.
Our main focus is on the case $k=2$: we set 
${\bf EB} := {\bf MB}^{/2}$.

\subsection{Emergent version of the Drinfeld-Kohno Lie algebra} \label{subsec:edk}

Let ${\sf c} = {\sf c}_{m,n}$ be the Lie ideal of $\dk_{m,n}$ generated by $c_{ij}$ for $1 \le i\neq j \le n$.

\begin{dfn}
The {\it emergent Drinfeld-Kohno Lie algebra} of type $(m,n)$ is the quotient Lie algebra
$
\edk_{m,n} := \dk_{m,n}/[{\sf c}, {\sf c}]
$.
\end{dfn}

\begin{remark}
Similarly, for each $k\ge 1$ one can define the quotient Lie algebra $\dk_{m,n}^{/k} := \dk_{m,n}/{\sf c}^{(k)}$, where ${\sf c}^{(k)}$ is the Lie ideal of $\dk_{m,n}$ inductively defined by ${\sf c}^{(1)} = {\sf c}$ and ${\sf c}^{(k)} = [{\sf c}^{(k-1)}, {\sf c}]$.
One has $\edk_{m,n} = \dk_{m,n}^{/2}$.
\end{remark}

In what follows we describe the structure of $\edk_{m,n}$.

\begin{lem} \label{lem:edk_direct}
We have a $\Q$-linear graded direct sum decomposition
\[
\edk_{m,n} \cong
\edk_{m,n-1} \oplus
\left( \lie(a_{1n},\ldots,a_{mn}) \oplus \bigoplus_{i=1}^{n-1} \ass(a_{1n},\ldots,a_{mn})[-1] \right).
\]
Here, $\ass(a_{1n},\ldots,a_{mn})[-1]$ is the degree shift of $\ass(a_{1n},\ldots,a_{mn})$ by $-1$: the constant term has degree $1$, the generators $x_1,\ldots, x_m$ have degree $2$, and so on.
\end{lem}

\begin{proof}
Let ${\sf c}_0$ be the Lie ideal of $\lie(a,c)$ generated by $c_{in}$, $1\le i \le n-1$.
Through the semi-direct decomposition of Lemma~\ref{lem:dk_decomp} the ideal $[{\sf c}_{m,n}, {\sf c}_{m,n}]$ corresponds to $[{\sf c}_{m,n-1}, {\sf c}_{m,n-1}] \oplus [{\sf c}_0, {\sf c}_0]$ in $\dk_{m,n-1} \oplus \lie(a,c)$, because ${\sf c}_{m,n} = {\sf c}_{m,n-1} \oplus {\sf c}_0$ and $[{\sf c}_{m,n-1}, {\sf c}_0] \subset [{\sf c}_0, {\sf c}_0]$.
Thus we obtain
\[
\edk_{m,n} \cong \edk_{m,n-1} \oplus \big( \lie(a,c)/[{\sf c}_0, {\sf c}_0] \big)
\]
as a $\Q$-linear space.
By the Lazard elimination theorem \cite[Chap II \S 2.9, Proposition 10]{Bou}, we have the following $\Q$-linear direct sum decomposition
\[
\lie(a,c) \cong \lie(a_{1n}, \ldots, a_{mn}) \oplus
\lie(\{ {\rm ad}_w (c_{in}) \}_{w,i}).
\]
Here, $\lie(\{ {\rm ad}_w (c_{in}) \}_{w,i})$ is the free Lie algebra generated by all elements of the form ${\rm ad}_w(c_{in}) = {\rm ad}_{w_1} \cdots {\rm ad}_{w_\lambda}(c_{in})$, where $1\le i \le n-1$ and $w = w_1 \cdots w_{\lambda}$ with  $w_1, \ldots, w_\lambda \in \{ a_{1n}, \ldots, a_{mn} \}$ runs over all associative words in $a_{1n}, \ldots, a_{mn}$ (including the empty word).
Hence
\[
\lie(a,c)/[{\sf c}_0, {\sf c}_0]
\cong
\lie(a_{1n}, \ldots, a_{mn}) \oplus
\bigoplus_{i=1}^{n-1} \bigoplus_{w}
\Q\, {\rm ad}_w(c_{in}).
\]
This proves the lemma.
\end{proof}

Repeated use of Lemma~\ref{lem:edk_direct} yields a $\Q$-linear graded direct sum decomposition 
\begin{equation} \label{eq:edk_decomp}
\edk_{m,n} \cong \bigoplus_{i=1}^n (\lie_m)_i
\oplus \bigoplus_{1 \le i < j \le n} (\ass_m[-1])_{ij},
\end{equation}
where the meaning of the components $(\lie_m)_i$ and $(\ass_m[-1])_{ij}$ is as follows:
\begin{align*}
(\lie_m)_i & \ni
u(x_1,\ldots,x_m)_i 
\mapsto u(a_{1i},\ldots,a_{mi}) \in \edk_{m,n}, \\
(\ass_m[-1])_{ij} &\ni
w(x_1,\ldots,x_m)_{ij}
\mapsto {\rm ad}_{w_j} (c_{ij}) \in \edk_{m,n}.
\end{align*}
Here, $u = u(x_1,\ldots,x_m) \in \lie_m$, $w = w(x_1,\ldots,x_m) \in \ass_m$ and we write $w_j = w(a_{1j},\ldots,a_{mj}) \in \ass(a_{1j},\ldots,a_{mj})$.

\begin{example}
\begin{enumerate}
\item[(i)]
$\edk_{2,1} \cong \lie_2$.
\item[(ii)]
$
\edk_{1,2} \cong (\lie_1)_1 \oplus (\lie_1)_2 \oplus (\ass_1[-1])_{12}
\cong \Q x_1 \oplus \Q x_2 \oplus \ass(x)[-1]
$.
\item[(iii)]
$
\edk_{2,2}
\cong \lie(x, y)_1 \oplus \lie(x, y)_2 \oplus (\ass(x,y)[-1])_{12}
$.
\end{enumerate}
\end{example}

In order to describe the Lie bracket on $\edk_{m,n}$ in view of the direct sum decomposition~\eqref{eq:edk_decomp}, we need to recall the partial differential operators on $\lie_m$ with respect to the generators $x_1, \ldots, x_m$.
Let $a \in \lie_m$.
Viewed as an element in $\ass_m$, it is uniquely written as 
\[
a = \sum_{i=1}^m (\pa_i a) x_i = \sum_{i=1}^m x_i (\pa^i a),
\]
where $\pa_i a, \pa^i a \in \ass_m$.
Furthermore, we have $\pa^i a = \iota(\pa_i a)$.
The operator $\pa_i = \pa_{x_i} : \lie_m \to \ass_m$ satisfies the following formula: for any $u, v \in \lie_m$, 
\begin{equation} \label{eq:pa([u,v])}
\pa_i([u,v]) = u (\pa_i v) - v (\pa_i u).
\end{equation}

The following proposition describes the Lie bracket on $\edk_{m,n}$.

\begin{prop} \label{prop:edk_bracket}
Let $u = u(x_1, \ldots, x_m), v = v(x_1,\ldots,x_m) \in \lie_m$ and $w = w(x_1, \ldots, x_m), w'=w'(x_1,\ldots,x_m) \in \ass_m$.
\begin{enumerate}
\item[{\rm (i)}]
For any $1 \le j \le n$, we have $[u_j, v_j] = [u,v]_j$.
For any $1 \le j < k \le n$,
\begin{equation} \label{eq:ujvk}
[u_j, v_k] = \left( \sum_{i=1}^m (\pa_i v) x_i \iota(\pa_i u) \right)_{jk}.
\end{equation} 
\item[{\rm (ii)}]
Let $1 \le i \le n$ and $1 \le j < k \le n$.
If $i \notin \{ j, k\}$, we have $[u_i, w_{jk}] = 0$.
Furthermore, we have $[u_k, w_{jk}] = (uw)_{jk}$ and $[u_j, w_{jk}] = -(wu)_{jk}$.

\item[{\rm (iii)}]
We have $[w_{ij}, w'_{kl}] = 0$ for any $1 \le i< j \le n$ and $1 \le k < l \le n$.
\end{enumerate}
\end{prop}

We need a lemma. 

\begin{lem} \label{lem:wkwj}
For $w = w(x_1,\ldots, x_m) \in \ass_m$ and $1 \le j \neq k \le n$, we have
\[
{\rm ad}_{w_k}(c_{jk})
= {\rm ad}_{\overline{w}_j}(c_{jk}).
\]
\end{lem}

\begin{proof}
We may assume that $w$ is a monomial of degree $d \ge 1$.
So let $w = x_{i_1} \cdots x_{i_d}$.
If $d = 1$, the formula holds true since $[a_{i_1 k}, c_{jk}] = - [a_{i_1 j}, c_{jk}]$.
Let $d \ge 2$ and assume that the formula holds true in degrees less than $d$.
Set $w' = x_{i_2} \cdots x_{i_d}$. 
Using the inductive assumption, we compute
\begin{align*}
& {\rm ad}_{a_{i_1 k}\cdots a_{i_d k}}(c_{jk})
= 
{\rm ad}_{a_{i_1 k}} {\rm ad}_{a_{i_2 k} \cdots a_{i_d k}}(c_{jk}) 
= {\rm ad}_{a_{i_1 k}} {\rm ad}_{\overline{a_{i_2 j}\cdots a_{i_d j}}}(c_{jk}) \\
=\, & (-1)^{d-1} \sum_{p=2}^d {\rm ad}_{a_{i_d j}} \cdots {\rm ad}_{[a_{i_1 k}, a_{i_p j}]} \cdots {\rm ad}_{a_{i_2 j}}(c_{jk}) + {\rm ad}_{\overline{a_{i_2 j}\cdots a_{i_d j}}}([a_{i_1 k},c_{jk}]).
\end{align*}
Since $[a_{i_1 k}, a_{i_p j}] = - \delta_{i_1 i_p} [c_{jk}, a_{i_p j}] \in {\sf c}$, the first term vanishes in $\edk_{m,n}$.
Therefore, ${\rm ad}_{a_{i_1 k}\cdots a_{i_d k}}(c_{jk})$ is equal to 
\[
{\rm ad}_{\overline{a_{i_2 j}\cdots a_{i_d j}}}([a_{i_1 k},c_{jk}])
= - {\rm ad}_{\overline{a_{i_2 j}\cdots a_{i_d j}}} [a_{i_1 j}, c_{jk}]
= {\rm ad}_{\overline{a_{i_1 j}\cdots a_{i_d j}}} (c_{jk}).
\]
This completes the proof.
\end{proof}

\begin{proof}[Proof of Proposition~\ref{prop:edk_bracket}]
First of all, the formula $[u_j, v_j] = [u,v]_j$ in (i) is clear.
In what follows, we will use this formula without mentioning explicitly.

(iii)
Since the expressions $w_{ij}$ and $w'_{kl}$ viewed as elements in $\dk_{m,n}$ are in the ideal ${\sf c}$, their commutator lies in $[{\sf c}, {\sf c}]$. Therefore $[w_{ij}, w'_{kl}] = 0 \in \edk_{m,n}$.

(ii)
To prove $[u_i, w_{jk}] = 0$ when $i \notin \{ j,k\}$, it is sufficient to consider the case where $u$ is of degree $1$ and $w$ is a monomial.
So we may assume that $u = x_q$ for some $1 \le q \le m$ and $w = x_{i_1}\cdots x_{i_d}$.
We compute
\begin{align*}
[u_i, w_{jk}] & = 
[a_{q i}, {\rm ad}_{w_k} (c_{jk})] \\
& = \sum_{p=1}^d {\rm ad}_{a_{i_1 k}} \cdots {\rm ad}_{[a_{q i}, a_{i_p k}]}\cdots {\rm ad}_{a_{i_d k}}(c_{jk})
+ {\rm ad}_{w_k}([a_{q i},c_{jk}]).
\end{align*}
The first term vanishes since $[a_{qi}, a_{i_p k}] = - \delta_{q i_p}[c_{ik}, a_{i_p k}] \in {\sf c}$.
The second term vanishes as well, since $[a_{q i}, c_{jk}] = 0$ by the commutation relation.

To prove the other two formulas, we first prove that
\begin{equation} \label{eq:uadw}
[u_k, {\rm ad}_{w_k}(c_{jk})] = {\rm ad}_{(uw)_k}(c_{jk})
\end{equation}
for any $1 \le i \le n$ and $1 \le j \neq k \le n$.
We may assume that $u$ is homogeneous and proceed by induction on $\deg u$.
When $\deg u = 1$, we have $[u_k, w_{jk}] = {\rm ad}_{u_k} {\rm ad}_{w_k}(c_{jk}) = {\rm ad}_{(uw)_k}(c_{jk}) = (uw)_{jk}$.
Let $\deg u > 2$ and assume that the formula holds true in degrees less than $\deg u$ and that $u$ is of the form $u = [u', u'']$.
We compute 
\begin{align*}
[u_k, {\rm ad}_{w_k}(c_{jk})] & =   
[[u'_k, {\rm ad}_{w_k}(c_{jk})], u''_k] + [u'_k, [u''_k, {\rm ad}_{w_k}(c_{jk})]] \\
& = [{\rm ad}_{(u'w)_k}(c_{jk}), u''_k] + [u'_k, {\rm ad}_{(u'' w)_k}(c_{jk})] \\ 
& = - {\rm ad}_{(u'' u' w)_k}(c_{jk}) + {\rm ad}_{(u' u'' w)_k}(c_{jk}) \\
& = {\rm ad}_{(uw)_k}(c_{jk}).
\end{align*}
In the second and third lines, we have used the inductive assumption.

Equation~\eqref{eq:uadw} shows that $[u_k,w_{jk}] = (uw)_{jk}$ for $j<k$.
To prove $[u_j, w_{jk}] = -(wu)_{jk}$ we compute
\[
[u_j, w_{jk}] = [u_j, {\rm ad}_{\overline{w}_j}(c_{jk})] = {\rm ad}_{u_j\overline{w}_j}(c_{jk}) = {\rm ad}_{w_k \overline{u}_k}(c_{jk}) = - {\rm ad}_{w_k u_k}(c_{jk}).
\]
Here, we have used Lemma~\ref{lem:wkwj} in the first and third equalities, formula \eqref{eq:uadw} in the second equality, and the fact that $\overline{u}_k = -u_k$ in the last equality.

(i) It remains to prove formula \eqref{eq:ujvk}.
Setting $\Phi(u,v) := \sum_{i=1}^m (\pa_i v) x_i \iota(\pa_i u)$, let us prove that $[u_j, v_k] = \Phi(u,v)_{jk}$ for any homogeneous elements $u, v \in \lie_m$.
We use the induction on the bidegree $(\deg u, \deg v)$.
Since $[a_{i_1 j}, a_{i_2 k}] = \delta_{i_1 i_2}[a_{i_2 k}, c_{jk}] = \delta_{i_1 i_2} (x_{i_2})_{jk}$, the case $(\deg u, \deg v) = (1,1)$ is done.
We first increase $\deg u$.
Let $\deg u > 1$ and assume that $u = [u',u'']$ for some $u', u'' \in \lie_m$ satisfying $[u'_j, v_k] = \Phi(u', v)_{jk}$ and $[u''_j, v_k]=\Phi(u'', v)_{jk}$. 
On the one hand, using these assumptions we compute
\begin{align*}
[u_j, v_k] & = [[u'_j, v_k], u''_j] + [u'_j, [u''_j, v_k]] \\
& = [\Phi(u',v)_{jk}, u''_j] + [u'_j, \Phi(u'',v)_{jk}] \\
& = \big( \Phi(u',v) u'' - \Phi(u'',v) u' \big)_{jk}.
\end{align*}
In the last line, we have used (ii).
On the other hand, using \eqref{eq:pa([u,v])} and the fact that $\iota$ acts as minus the identity on $\lie_m$, we see that $\Phi(u,v) = \Phi([u',u''],v) = \Phi(u',v)u'' - \Phi(u'',v)u'$.
Hence we conclude that $[u_j, v_k] = \Phi(u,v)_{jk}$. 
A similar argument works for increasing $\deg v$.
This completes the proof.
\end{proof}

Let $\mathcal{A}_{m,n}^\EM = U(\edk_{m,n})$ be the universal enveloping algebra of $\edk_{m,n}$.
It is the quotient of $\mathcal{A}_{m,n} = U(\dk_{m,n})$ by the span of monomials in $a_{ij}$ and $c_{ij}$ which contain at least two generators of type $c_{ij}$. 

The augmentation ideal of $\Q P_{m,n}$ projects to a two-sided ideal $\bar{I}$ of $\Q P_{m,n}^\EM$.
Let $\mathcal{A}_{\Q P_{m,n}^\EM}:= \prod_{k \ge 0} {\bar{I}}^k / {\bar{I}}^{k+1}$ be the completed associated graded algebra.
The following proposition is a consequence of Proposition~\ref{prop:grpureemb}.

\begin{prop} \label{prop:grpuremb2}
There is a canonical isomorphism of graded $\Q$-algebras $\mathcal{A}_{\Q P_{m,n}^\EM} \cong \mathcal{A}_{m,n}^\EM$, through which the class of $\alpha_{ij}-1$ (resp. $\tau_{ij}-1$) corresponds to $a_{ij}$ (resp. $c_{ij}$).
\end{prop}

\subsection{Description of operadic operations on $\edk_{m,n}$} \label{subsec:dooedk}

The operadic operations introduced in Section~\ref{subsec:cfc_mps} naturally induce operations on emergent braids and chord diagrams.
Let us describe the operations on $\edk_{m,n}$ in view of the direct sum decomposition \eqref{eq:edk_decomp}.
In what follows, let $u = u(x_1,\ldots,x_m) \in \lie_m$ and $w = w(x_1,\ldots, x_m) \in \ass_m$. 

First, we have 
\begin{equation} \label{eq:del0}
\delta_0(u_i) = u(x_2, \ldots, x_{m+1})_i ,
\qquad  
\delta_0(w_{ij}) = w(x_2,\ldots,x_{m+1})_{ij},
\end{equation}
and for $1 \le k \le m$, 
\begin{align}
\delta_k( u_i) &= u(x_1,\ldots,x_k + x_{k+1}, \ldots, x_{m+1})_i, \nonumber \\
\delta_k(w_{ij}) &= w(x_1,\ldots,x_k + x_{k+1}, \ldots, x_{m+1})_{ij}. \label{eq:delk}
\end{align}

Second, we describe the cabling operations with respect to blue strands. 
Let $R : \lie_m \to \ass_m$ be the unique $\Q$-linear map satisfying $R(x_i) = 0$ for $i=1,\ldots,m$ and for any $a,b \in \lie_m$, 
\begin{equation} \label{eq:Rrecursive}
R([a,b]) = [R(a), b] + [a, R(b)] 
+ \sum_{i=1}^m \left( (\pa_i b )x_i \iota(\pa_i a)
- (\pa_i a) x_i \iota(\pa_i b) \right).
\end{equation}

\begin{lem} \label{lem:edk_cab_strnd}
For $1 \le k \le n$, we have the following: 
\[
\delta_{m+k}(u_i) =  
\begin{cases}
u_i & (i < k) \\
u_k + u_{k+1} + R(u)_{k(k+1)} & (i = k) \\
u_{i+1} & (i > k)
\end{cases},
\]
and 
\[
\delta_{m+k}(w_{ij}) =
\begin{cases}
w_{ij} & ( j < k) \\
w_{ik} + w_{i(k+1)} & ( j = k) \\
w_{i(j+1)} & ( i < k < j) \\
w_{k(j+1)} + w_{(k+1)(j+1)} & ( i = k) \\
w_{(i+1)(j+1)} & ( k < i)
\end{cases}.
\]
\end{lem}

\begin{proof}
We will prove the formula $\delta_{m+k}(u_k) = u_k + u_{k+1} + R(u)_{k(k+1)}$ and $\delta_{m+k}(w_{ik}) = w_{ik} + w_{i(k+1)}$ only. 
The proof of the other formulas is rather straightforward, so we omit it.

First we prove that $\delta_{m+k}(u_k) = u_k + u_{k+1} + R(u)_{k(k+1)}$.
This is true in degree one, since $\delta_{m+k}(a_{ik}) = a_{ik} + a_{i(k+1)}$.
Assume that $\deg u > 1$, we have $u = [a, b]$ for some homogeneous elements $a, b$, and
\[
\delta_{m+k}(a_k) = a_k + a_{k+1} + R(a)_{k(k+1)}, 
\quad
\delta_{m+k}(b_k) = b_k + b_{k+1} + R(b)_{k(k+1)}
\]
for some $R(a), R(b) \in \ass_m$.
Then, we have 
\begin{align*}
\delta_{m+k}(u_k) &= [\delta_{m+k}(a_k), \delta_{m+k}(b_k)] \\
&= [a_k + a_{k+1} + R(a)_{k(k+1)}, b_k + b_{k+1} + R(b)_{k(k+1)}].
\end{align*}
Computing the right hand side using Proposition~\ref{prop:edk_bracket}, we obtain
\[
u_k + u_{k+1} + \big(
[R(a), b] + [a, R(b)] + 
\sum_{i=1}^m \left( (\pa_i b) x_i \iota(\pa_i a) - (\pa_i a) x_i \iota(\pa_i b) \right) \big)_{k(k+1)}.
\]
This completes the proof.

Next we show that $\delta_{m+k}(w_{ik}) = w_{ik} + w_{i(k+1)}$.
We have
\begin{align*}
\delta_{m+k}(w_{ik}) &= 
{\rm ad}_{w(a_{1k}+a_{1(k+1)},\ldots,a_{mk}+a_{m(k+1)})}(c_{ik}+c_{i(k+1)}) \\
&= {\rm ad}_{w(a_{1k}+a_{1(k+1)},\ldots,a_{mk}+a_{m(k+1)})}(c_{ik}) \\
& \qquad + {\rm ad}_{w(a_{1k}+a_{1(k+1)},\ldots,a_{mk}+a_{m(k+1)})}(c_{i(k+1)}).
\end{align*}
Since $a_{j(k+1)}$ and $c_{ik}$ commute and the Lie bracket of $a_{jk}$ and $a_{l(k+1)}$ lies in ${\sf c}$, the first term is equal to ${\rm ad}_{w(a_{1k},\ldots,a_{mk})}(c_{ik}) = w_{ik}$.
Similarly, the second term is equal to $w_{i(k+1)}$. 
This completes the proof.
\end{proof}

Finally, we describe the map $\vartheta_m$.

\begin{lem} \label{lem:edk_taum}
We have the following:
\begin{align*}
\vartheta_m(u_i) &= u(x_1,\ldots,x_{m-1},0)_{i+1} + \big( (\pa_m  u)(x_1,\ldots,x_{m-1},0) \big)_{1(i+1)}, \\
\vartheta_m(w_{ij}) &= w(x_1,\ldots,x_{m-1},0)_{(i+1)(j+1)}.
\end{align*}
\end{lem}

\begin{proof}
The proof of the first formula is similar to the proof of the formula for $\delta_{m+k}(u_k)$ in Lemma~\ref{lem:edk_cab_strnd}.
We denote by $H(u)$ the right hand side of the formula.
We first check that the formula holds true in degree one.
Now let $a, b \in \lie_m$ and assume that $\vartheta_m(a_i) = H(a)$ and $\vartheta_m(b_i) = H(b)$.
Then, by a direct computation using Proposition~\ref{prop:edk_bracket} and formula~\eqref{eq:pa([u,v])}, we verify that 
$\vartheta_m([a,b]_i) = [H(a), H(b)]$ is equal to $H([a,b])$.
Since this is straightforward, we omit the detail.

To prove the second formula, modulo $[{\sf c}, {\sf c}]$ we compute
\begin{align*}
\vartheta_m(w_{ij}) &= {\rm ad}_{w(a_{1(j+1)},\ldots,a_{(m-1)(j+1)},c_{1(j+1)})}(c_{(i+1)(j+1)}) \\
&= {\rm ad}_{w(a_{1(j+1)},\ldots,a_{(m-1)(j+1)},0)}(c_{(i+1)(j+1)}) \\
&= w(x_1,\ldots,x_{m-1},0)_{(i+1)(j+1)}.
\qedhere
\end{align*}
\end{proof}

\section{Homomorphic expansions for mixed braids} \label{sec:hexpmb}

In \cite{BN98}, the category ${\bf PaB}$ of parenthesized braids was introduced, and it was shown that the Drinfeld associators give rise to $1$-formality isomorphisms (homomorphic expansions) for this category.
In this section, we extend this formalism to mixed braids.

\subsection{Parenthesized mixed braids and chord diagrams} \label{subsec:parem}

We need some notation.
Let ${\bf Par}_{\rb \bb} = \bigsqcup_{m,n \ge 0} {\bf Par}_{m,n}$ be the set of parenthesized words in two letters $\rb$ and $\bb$, where ${\bf Par}_{m,n}$ is the subset consisting of parenthesized words with $m$ red bullets and $n$ blue bullets.
For example, $(\rb \rb)\bb \in {\bf Par}_{2,1}$ and $\rb(\rb(\bb \rb)) \in {\bf Par}_{3,1}$.
For $O \in {\bf Par}_{m,n}$, let $f(O)=\overline{O} \in {\bf Par}_{m,0}$ be the parenthesized word in $\rb$ obtained by forgetting all the blue bullets in $O$, and let $p(O) = o \in W_{m,n}$ be the word obtained by forgetting the parenthesization of $O$.
For example, if $O = \rb(\rb(\bb \rb))$, then
$\overline{O} = \rb(\rb \rb)$ and $o = \rb \rb \bb \rb$.

First we define the category ${\bf PaMB}$ of parenthesized mixed braids.
The set of objects is ${\bf Par}_{\rb \bb}$.
Let $O, O' \in {\bf Par}_{\rb \bb}$ with $f(O) = \overline{O}$, $f(O') = \overline{O'}$, $p(O) = o$ and $p(O') =o'$.
Then the set of morphisms from $O$ to $O'$ is
\[
{\bf PaMB}(O, O') := \begin{cases}
{\bf MB}(o,o') & \text{if $\overline{O} = \overline{O'}$}, \\
\emptyset & \text{otherwise}.
\end{cases}
\]
The composition is defined using that of ${\bf MB}$.
Note that there are no morphisms from $O$ to $O'$ unless $\overline{O} = \overline{O'}$.
For example, we have no morphism from $(\rb \rb) \rb$ to $\rb(\rb \rb)$.
When we draw pictures of morphisms in ${\bf PaMB}$, which are represented by linear combinations of mixed braids, we use the same convention used for ${\bf PaB}$ in \cite{BN98}.
Namely, we draw the bottom and top ends of mixed braids so that their distances respect their ``distances'' in the parenthesization of the source and domain of the morphism. 

\begin{example} \label{exple:pamor}
In the following two pictures, the first one shows a morphism from $(\rb \rb)\bb$ to $\rb (\rb \bb)$, and the second one from $\rb(\bb(\bb \rb))$ to $(\bb \rb)(\bb \rb)$.
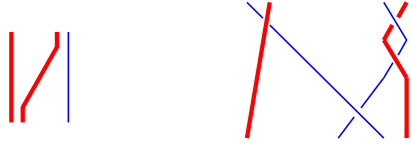
\[
 \begin{tikzpicture}[baseline=12pt, x=1.5mm, y=2mm]
\draw[rs] (0,0) -- (0,6);
\draw[rs] (1,0) -- (1,1) -- (4,5) -- (4,6);
\draw[bs] (5,0) -- (5,6);
\end{tikzpicture}
\hspace{6em}
\begin{tikzpicture}[baseline=18pt, x=1.5mm, y=2mm]
\draw[rs] (0,0) -- (2,9);
\draw[bs] (8,0) -- (9.4, 1.4);
\draw[bs] (10,2) -- (12,4);
\draw[bs] (12,0) -- (2,7.5);
\draw[bs] (2-0.6, 7.5+0.45) -- (0,9);
\draw[rs] (14,0) -- (14,4);
\draw[bs] (12,4) -- (13-0.2,5.25-0.25);
\draw[bs] (13+0.25,5.25+0.25) -- (14,6.5);
\draw[bs] (14,6.5) -- (12,9);
\draw[rs] (14,4) -- (12,6.5);
\draw[rs] (12,6.5) -- (13-0.25,7.75-0.25);
\draw[rs] (13+0.25,7.75+0.25) --(14,9);
\end{tikzpicture}
\]
\end{example}

Next we define the category ${\bf PaMCD}$.
The set of objects is the same as the set of objects of ${\bf PaMB}$, namely ${\bf Par}_{\rb \bb}$.
The set of morphisms from $O$ to $O'$ is
\[
{\bf PaMCD}(O,O') := \begin{cases}
\mathcal{A}_{m,n} \times \mathfrak{S}_{\rb \bb}(o,o') & \text{if $\overline{O} = \overline{O'}$}, \\
\emptyset & \text{otherwise}.
\end{cases}
\]
Here, in the first case, $(m,n)$ is the type of $o \in W_{m,n}$.
By definition, a morphism in ${\bf PaMCD}$ is of the form $(u,\sigma)$, where $u \in \mathcal{A}_{m,n}$ and $\sigma$ is a mixed permutation of type $(m,n)$.
Recall that $u$ is expressed as a linear combination of words in $a_{ij}$ and $c_{ij}$, which are interpreted as horizontal chords.
We draw $u$ on the picture of $\sigma$ so that $a_{ij}$ (resp. $c_{ij}$) becomes a chord connecting the $i$th red strand and $j$th blue strand (resp. the $i$th and $j$th blue strands), where we count the strands at the bottom.
We also express the information on the parenthesization using the distances between endpoints.
For example, 
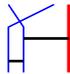
\[
\begin{tikzpicture}[x=2mm, y=3mm]
\draw[bs] (0,0) -- (0,2) -- (3,3);
\draw[bs] (1,0) -- (1,2) -- (0,3);
\draw[rs] (4,0) -- (4,2) -- (4,3);
\draw (0,0.5) -- (1,0.5);
\draw (1,1.5) -- (4,1.5);
\end{tikzpicture}
\]
is a morphism from $(\bb \bb)\rb$ to $\bb (\bb \rb)$ which corresponds to
$\big( c_{12}a_{12}, 
\begin{tikzpicture}[baseline=3pt, x=2mm, y=2mm]
\draw[bs] (0,0) -- (1,2);
\draw[bs] (1,0) -- (0,2);
\draw[rs] (2,0) -- (2,2);
\end{tikzpicture} \, 
\big)
$.
In this view point, the composition in ${\bf PaMCD}$ is given by stacking of diagrams. 
For example, one has 
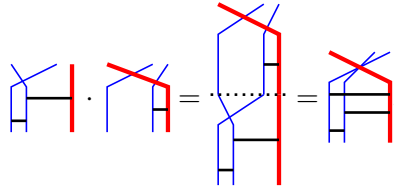
\[
\begin{tikzpicture}[baseline=10pt, x=2mm, y=3mm]
\draw[bs] (0,0) -- (0,2) -- (3,3);
\draw[bs] (1,0) -- (1,2) -- (0,3);
\draw[rs] (4,0) -- (4,2) -- (4,3);
\draw (0,0.5) -- (1,0.5);
\draw (1,1.5) -- (4,1.5);
\end{tikzpicture}
\, \cdot \,
\begin{tikzpicture}[baseline=10pt, x=2mm, y=3mm]
\draw[bs] (0,0) -- (0,2) -- (3,3);
\draw[bs] (3,0) -- (3,2) -- (4,3);
\draw[rs] (4,0) -- (4,2) -- (0,3);
\draw (3,1) -- (4,1);
\end{tikzpicture}
= 
\begin{tikzpicture}[baseline=30pt, x=2mm, y=4mm]
\draw[bs] (0,0) -- (0,2) -- (3,3);
\draw[bs] (1,0) -- (1,2) -- (0,3);
\draw[rs] (4,0) -- (4,2) -- (4,3);
\draw (0,0.5) -- (1,0.5);
\draw (1,1.5) -- (4,1.5);
\draw[bs] (0,3) -- (0,5) -- (3,6);
\draw[bs] (3,3) -- (3,5) -- (4,6);
\draw[rs] (4,3) -- (4,5) -- (0,6);
\draw (3,4) -- (4,4);
\draw[dotted] (-0.5, 3) -- (4.5, 3);
\end{tikzpicture}
= 
\begin{tikzpicture}[baseline=14pt, x=2mm, y=4mm]
\draw[bs] (0,0) -- (0,2) -- (4,3);
\draw[bs] (1,0) -- (1,2) -- (3,3);
\draw[rs] (4,0) -- (4,2) -- (0,3);
\draw (0,0.4) -- (1,0.4);
\draw (1,1) -- (4,1);
\draw (0,1.6) -- (4,1.6);
\end{tikzpicture}.
\]
More formally, the composition of (composable) morphisms $(u, \sigma)$ and $(u', \sigma')$ is given by $(u \sigma(u'), \sigma \sigma')$.
Here, through the restriction to the blue bullets, $\sigma$ induces a permutation of $\{ 1, \ldots, n\}$ and hence acts on $\mathcal{A}_{m,n}$.

The operadic operations to mixed braids and chord diagrams introduced in Section~\ref{subsec:cfc_mps} extend naturally to their parenthesized enhancements ${\bf PaMB}$ and ${\bf PaMCD}$, with an extra care for parenthesizations.
For the extension operations, we draw the ends of the added strand outer-most in the picture.
For the cabling operations, we draw the ends of the two newly created strands closest to each other.
For example, 
\[
\delta_0^s \left( \, 
\begin{tikzpicture}[baseline=8pt, x=4mm, y=4mm]
\pc{rs}{bs}{0,0}
\end{tikzpicture}
\, \right) = 
\begin{tikzpicture}[baseline=8pt, x=3mm, y=4mm]
\pc{rs}{bs}{0,0}
\draw[bs] (-2,0) -- (-2,2);
\end{tikzpicture}
\quad \text{and} \quad
\delta_1 \left( \, 
\begin{tikzpicture}[baseline=8pt, x=4mm, y=4mm]
\pc{rs}{bs}{0,0}
\end{tikzpicture}
\, \right) = 
\begin{tikzpicture}[baseline=10pt, x=4.5mm, y=4.5mm]
\draw[rs] (0.25,0) -- (0.25,0.6) -- (1,1.35) -- (1,2);
\draw[rs] (0,0) -- (0,0.7) -- (0.75,1.45) -- (0.75,2);
\draw[bs] (1,0) -- (1,0.5) -- (0.75, 0.75);
\draw[bs] (0,2) -- (0,1.5) -- (0.25,1.25); 
\end{tikzpicture} \, .
\]

To compare the categories ${\bf PaMB}$ and ${\bf PaMCD}$, we need to consider their completions.
On the one hand, ${\bf PaMB}$ is filtered.
By the same argument used for the ideal $J$ in Section~\ref{subsec:eb}, the augmentation ideal $I = I P_{m,n}$ and its powers extend naturally to a multiplicative filtration of the category ${\bf MB}$ and hence of ${\bf PaMB}$.
Therefore, one can define the $I$-adic completion $\widehat{{\bf PaMB}}$ and the associated graded ${\rm gr}\, {\bf PaMB}$.
On the other hand, ${\bf PaMCD}$ is graded.
The grading comes from the grading of the algebra $\mathcal{A}_{m,n}$.
Thus one can define the degree completion $\widehat{{\bf PaMCD}}$, where $\mathcal{A}_{m,n}$ is replaced with its degree completion $\widehat{\mathcal{A}}_{m,n}$.
The operadic operations on ${\bf PaMB}$ and ${\bf PaMCD}$ extend to their completions.

The isomorphism $\mathcal{A}_{\Q P_{m,n}} \cong \mathcal{A}_{m,n}$ proven in Proposition~\ref{prop:grpureemb} extends naturally to a canonical isomorphism
\begin{equation} \label{eq:pambmcd}
{\rm gr}\, {\bf PaMB} \cong {\bf PaMCD}
\end{equation}
which is the identity on the objects.
Moreover, one checks that this isomorphism respects all the operadic operations.

\begin{prop} \label{prop:paeb_gen}
The category ${\bf PaMB}$ is generated by the following morphisms, their inverses, and their images by repeated applications of the operadic operations in ${\bf PaMB}$:
\begin{equation} \label{eq:gen_paeb}
\sigma_{\rm ps}^+ = 
\begin{tikzpicture}[baseline=8.5pt, x=4mm, y=4mm]
\pc{rs}{bs}{0,0}
\end{tikzpicture},
\quad 
\sigma_{\rm ps}^- = 
\begin{tikzpicture}[baseline=8.5pt, x=4mm, y=4mm]
\nc{rs}{bs}{0,0}
\end{tikzpicture},
\quad
\alpha_{\rm pps} := 
\begin{tikzpicture}[baseline=14pt, x=1.5mm, y=2mm]
\draw[rs] (0,0) -- (0,6);
\draw[rs] (1,0) -- (1,1) -- (4,5) -- (4,6);
\draw[bs] (5,0) -- (5,6);
\end{tikzpicture},
\quad
\alpha_{\rm psp} :=
\begin{tikzpicture}[baseline=14pt, x=1.5mm, y=2mm]
\draw[rs] (0,0) -- (0,6);
\draw[bs] (1,0) -- (1,1) -- (4,5) -- (4,6);
\draw[rs] (5,0) -- (5,6);
\end{tikzpicture},
\quad
\alpha_{\rm spp} :=
\begin{tikzpicture}[baseline=14pt, x=1.5mm, y=2mm]
\draw[bs] (0,0) -- (0,6);
\draw[rs] (1,0) -- (1,1) -- (4,5) -- (4,6);
\draw[rs] (5,0) -- (5,6);
\end{tikzpicture}.
\end{equation}
\end{prop}

\begin{proof}[Sketch of proof]
We say that a parenthesized mixed braid is basic if it is one of the morphisms listed in the statement of the proposition.
Let $m, n \ge 0$ and $\overline{O}^*_m \in {\bf Par}_{m,0}$.
Joining $\overline{O}^*_m$ and the left-nested parenthesization of $n$ blue bullets, we obtain a parenthesized word $O^*_{m,n} \in {\bf Par}_{m,n}$.
For example, if $\overline{O}^*_4 = (\rb \rb) (\rb \rb)$ then $O^*_{4,3} = (( \rb \rb) (\rb \rb))( ( \bb \bb ) \bb)$.

Let $m, n \ge 0$ and $O, O' \in {\bf Par}_{m,n}$ such that $\overline{O}^*_m := \overline{O} = \overline{O'}$.
Given any parenthesized mixed braids $\beta \in {\bf PaMB}(O^*_{m,n}, O)$ and $\beta' \in {\bf PaMB}(O^*_{m,n}, O')$, any morphism $\xi$ from $O$ to $O'$ decomposes as $\xi = \beta^{-1} ( \beta \xi {\beta'}^{-1}) \beta'$.
Therefore, to prove the proposition it is sufficient to show the following: given any parenthesization $\overline{O}^*_m \in {\bf Par}_{m,0}$,
\begin{enumerate}
\item[(i)] any parenthesized mixed braid from $O^*_{m,n}$ to itself decomposes into a product of basic morphisms, and
\item[(ii)] for any $O \in {\bf Par}_{m,n}$ with $\overline{O} = \overline{O}^*_m$, there is morphism from $O^*_{m,n}$ to $O$ which decomposes into a product of basic morphisms.
\end{enumerate}

To prove (i), note that the underlying mixed braid of a parenthesized one lies in the group $B_{m,n}^{\rm std}$ introduced in Section~\ref{subsec:mb}.
Since this group is generated by $\alpha_{ij}$'s and the simple braids among blue strands \cite[Section~4]{Lam00}, it is sufficient to deal with these generators.
We give two sample computations: 
\[
\begin{tikzpicture}[baseline=14pt, x=1.5mm, y=2mm]
\draw[rs] (-1,0) -- (-1,6);
\draw[bs] (6,0) -- (6,6);
\draw[bs] (7,0) -- (11,6);
\draw[bs] (11,0) -- (9.2,2.7);
\draw[bs] (8.8,3.3) -- (7,6);
\end{tikzpicture} \, = 
\, 
\begin{tikzpicture}[baseline=26pt, x=1.5mm, y=2.5mm]
\draw[rs] (-1,0) -- (-1,8);
\draw[bs] (6,0) -- (6,8);
\draw[bs] (7,0) -- (7,0.5) -- (10,2.5) -- (10,3);
\draw[bs] (11,0) -- (11,3);
\pc{bs}{bs}{10,3}
\draw[bs] (10,5) -- (10,5.5) -- (7,7.5) -- (7,8);
\draw[bs] (11,5) -- (11,8);
\draw[dotted] (-2, 3) -- (12, 3);
\draw[dotted] (-2, 5) -- (12, 5);
\end{tikzpicture}
= \delta_0^p (\alpha_{\rm sss}) \circ \delta_0^p (\delta_0^s (\sigma_{\rm ss})) 
\circ \delta_0^p (\alpha_{\rm sss}^{-1}),
\]
where we set $\sigma_{\rm ss} = \vartheta_1(\sigma_{\rm ps}^+)$ and $\alpha_{\rm sss} = \vartheta_1 (\vartheta_2 ( \alpha_{\rm pps}))$, and 
\[
\begin{tikzpicture}[baseline=14pt, x=1.5mm, y=2mm]
\draw[rs] (0,0) -- (0,3.4);
\draw[rs] (0,3.7) -- (0,6);
\draw[rs] (1,0) -- (1,2.05);
\draw[rs] (1,2.35) -- (1,3.75);
\draw[rs] (1,4.05) -- (1,6);
\draw[bs] (5,0) -- (5,0.7);
\draw[bs] (5,1) -- (5,5);
\draw[bs] (5,5.3) -- (5,6);
\draw[bs] (-0.4,2.6) to[bend left=90] (-0.4,3.4);
\draw[bs] (0.4,2.4) -- (6,0.5) -- (6,0);
\draw[bs] (-0.42,3.41) -- (6,5.5) -- (6,6);
\end{tikzpicture}
= 
\begin{tikzpicture}[baseline=48pt, x=1.5mm, y=1.5mm]
\draw[rs] (0,0) -- (0,10);
\draw[rs] (0,14) -- (0,24);
\draw[rs] (1,0) -- (1,4);
\draw[rs] (1,4) -- (1,4.2) -- (4,5.8) -- (4,6);
\draw[bs] (1,14) -- (1,14.2) -- (4,15.8) -- (4,16);
\draw[rs] (1,20) -- (1,24);
\draw[bs] (4,8) -- (4,8.2) -- (1,9.8) -- (1,10);
\draw[rs] (4,18) -- (4,18.2) -- (1,19.8) -- (1,20);
\draw[bs] (5,4) -- (5,6);
\draw[rs] (5,8) -- (5,16);
\draw[bs] (5,18) -- (5,20);
\draw[bs] (5,20) -- (5,20.2) -- (11,21.8) -- (11,22);
\draw[bs] (11,2) -- (11,2.2) -- (5,3.8) -- (5,4);
\draw[bs] (12,2) -- (12,22);
\nc{bs}{bs}{11,0}
\pc{rs}{bs}{0,10}
\pc{bs}{rs}{0,12}
\nc{rs}{bs}{4,6}
\pc{bs}{rs}{4,16}
\pc{bs}{bs}{11,22}
\draw[dotted] (-1,2) -- (13,2);
\draw[dotted] (-1,4) -- (13,4);
\draw[dotted] (-1,6) -- (13,6);
\draw[dotted] (-1,8) -- (13,8);
\draw[dotted] (-1,10) -- (13,10);
\draw[dotted] (-1,12) -- (13,12);
\draw[dotted] (-1,14) -- (13,14);
\draw[dotted] (-1,16) -- (13,16);
\draw[dotted] (-1,18) -- (13,18);
\draw[dotted] (-1,20) -- (13,20);
\draw[dotted] (-1,22) -- (13,22);
\end{tikzpicture}
= \xi \circ \delta_4^s (\delta_3^p (\sigma_{\rm ps}^+)) \circ \delta_4^s (\delta_3^p (\sigma_{\rm ps}^-)) \circ \xi^{-1},
\]
where $\xi = \delta_1(\delta_0^p(\sigma_{\rm ss}^{-1}))
\circ \delta_1( \vartheta_2(\alpha_{\rm pps}^{-1}))
\circ \delta_4^s(\alpha_{\rm pps}) \circ \delta_4^s(\delta_0^p(\sigma_{\rm ps}^-)) \circ \delta_4^s(\alpha_{\rm psp}^{-1})$.

For (ii), we give one example: 
\[
\begin{tikzpicture}[baseline=14pt, x=1.5mm, y=2mm]
\draw[bs] (0,6) -- (0.6,5.28);
\draw[bs] (1.05,4.84) -- (2.55,2.94);
\draw[bs] (2.95,2.46) -- (5,0);
\draw[rs] (1,6) -- (0,0);
\draw[rs] (5,6) -- (1,0);
\draw[bs] (6,6) -- (6,0);
\end{tikzpicture}
=
\begin{tikzpicture}[baseline=32pt, x=1.25mm, y=3mm]
\draw[bs] (0,8) -- (0,4);
\draw[rs] (1,8) -- (1,6);
\draw[rs] (9,8) -- (9,7.8) -- (4,6.2) -- (4,6);
\draw[bs] (10,8) -- (10,0);
\draw[rs] (1,6) -- (1,5.8) -- (3,4.2) -- (3,4);
\draw[rs] (4,6) -- (4,4);
\draw[rs] (3,4) -- (3,3.8) -- (0,2.4) -- (0,2);
\draw[rs] (4,4) -- (4,3.8) -- (1,2.4) -- (1,2);
\draw[bs] (0,4) -- (1.4,3.3);
\draw[bs] (2.4,2.8) -- (4,2);
\draw[rs] (0,2) -- (0,0);
\draw[rs] (1,2) -- (1,0);
\draw[bs] (4,2) -- (4,1.8) -- (9,0.2) -- (9,0);
\draw[dotted] (-1,2) -- (11,2);
\draw[dotted] (-1,4) -- (11,4);
\draw[dotted] (-1,6) -- (11,6);
\end{tikzpicture}
= \delta_1(\vartheta_2(\alpha_{\rm pps}^{-1})) \circ
\delta_4^s(\delta_2(\sigma_{\rm ps}^+)) \circ
\delta_4^s(\alpha_{\rm spp}^{-1}) \circ 
\vartheta_1(\delta_1(\alpha_{\rm pps})). 
\]
\end{proof}

In a similar fashion to the definition of ${\bf PaMB}$, for each $k \ge 1$ we define the parenthesized version of the category ${\bf MB}^{/k}$, which we denote by ${\bf PaMB}^{/k}$.
Likewise, for each $k \ge 1$ we define the category ${\bf PaMCD}^{/k}$.
The isomorphism \eqref{eq:pambmcd} descends to an isomorphism
$
{\rm gr}\, {\bf PaMB}^{/k} \cong {\bf PaMCD}^{/k}
$.
For $k=2$, we use the special notation ${\bf PaEB} = {\bf PaMB}^{/2}$ and ${\bf PaECD} = {\bf PaMCD}^{/2}$.

\subsection{Homomorphic expansions for mixed braids} \label{subsec:homo_expansion}

Here comes the concept of homomorphic expansions for the category ${\bf PaMB}$:

\begin{dfn} \label{def:expansion}
A {\it homomorphic expansion} for ${\bf PaMB}$ is a functor $Z^{\rm mb} : {\bf PaMB} \to \widehat{{\bf PaMCD}}$ which is the identity on the objects, preserves the filtrations, induces the identity at the associated graded, respects all the operadic operations, and is group-like.
\end{dfn}

The group-like condition in the above definition means that for each mixed braid $\beta$ of type $(m,n)$ one has $Z^{\rm mb}(\beta) = ( \exp(u), \pi(\beta))$, where $u$ is an element in $\widehat{\dk}_{m,n}$, the degree completion of $\dk_{m,n}$.

For each $k\ge 1$, we can formulate the concept of a homomorphic expansion for ${\bf PaMB}^{/k}$: it is defined to be a functor $Z^{{\rm mb}/k}: {\bf PaMB}^{/k} \to \widehat{{\bf PaMCD}^{/k}}$ satisfying the same conditions required for $Z^{\rm mb}$ in Definition~\ref{def:expansion}. 

Homomorphic expansions for ${\bf PaMB}$ exist by the following

\begin{prop} \label{prop:ass}
Any Drinfeld associator gives rise to a homomorphic expansion for ${\bf PaMB}$ and consequently to a homomorphic expansion for ${\bf PaMB}^{/k}$ for any $k \ge 1$.
\end{prop}

\begin{proof}
Let $\Phi$ be a Drinfeld associator. 
It is of the form $\Phi = \exp(\varphi)$, where $\varphi = \varphi(x,y) \in \widehat{\lie}_2$ is a Lie series without linear term.
As shown in \cite[Proposition~3.4]{BN98}, $\Phi$ extends to a functor $Z^{\rm pb} : {\bf PaB} \to \widehat{\bf PaCD}$, where the target is the degree completion of the category of parenthesized (horizontal) chord diagrams.
The category ${\bf PaB}$ is generated (in the same sense as in Proposition~\ref{prop:paeb_gen}) by the elements
$\sigma = 
\begin{tikzpicture}[baseline=3pt, x=2mm, y=2mm]
\pc{black}{black}{0,0}
\end{tikzpicture}$
and  
$\alpha = 
\begin{tikzpicture}[baseline=8pt, x=0.9mm, y=1.2mm]
\draw (0,0) -- (0,6);
\draw (1,0) -- (1,1) -- (4,5) -- (4,6);
\draw (5,0) -- (5,6);
\end{tikzpicture}$
(see \cite[Claim~2.6]{BN98}).
The functor $Z^{\rm pb}$ is specified by the values on these generators:
\[
Z^{\rm pb}( \sigma) :=  \left( \exp \left( \frac{1}{2} t_{12} \right),
\begin{tikzpicture}[baseline=8pt, x=1.8mm, y=3.6mm]
\draw (0,0) -- (2,2);
\draw (2,0) -- (0,2);
\end{tikzpicture}\, \right),
\qquad
Z^{\rm pb}(\alpha):= \left( \Phi(t_{12}, t_{23}), 
\begin{tikzpicture}[baseline=8pt, x=0.9mm, y=1.2mm]
\draw (0,0) -- (0,6);
\draw (1,0) -- (4,6);
\draw (5,0) -- (5,6);
\end{tikzpicture}\, \right).
\]

There are functors ${\bf PaMB} \to {\bf PaB}$ and $\widehat{{\bf PaMCD}} \to \widehat{{\bf PaCD}}$ obtained by forgetting the colors of poles and strands.
These functors are faithful.
For the former, this follows from the fact that $P_{m,n}$ is a subgroup of $P_{m+n}$.
For the latter, this follows from the injectivity of the map $\dk_{m,n} \to \dk_{m+n}$; see Remark~\ref{rem:dkmninj}.

We will show that $Z^{\rm pb}$ induces a functor $Z^{\rm mb} : {\bf PaMB} \to \widehat{{\bf PaMCD}}$ which is the identity on the objects.
In view of the faithfulness of the forgetful functors above, it is sufficient to prove the following claim:

\begin{claim}
Let $\beta$ be a mixed braid with $m$ poles and $n$ strands which represents a morphism in ${\bf PaMB}$.
Let $\beta^0$ be the parenthesized braid obtained by forgetting the colors of $\beta$ and write $Z^{\rm pb}(\beta^0) = (B, \pi(\beta^0))$, where $B \in \exp(\widehat{\dk}_{m+n})$ and $\pi(\beta^0)$ is the parenthesized permutation induced by $\beta^0$.
Then, $B$ lies in $\exp(\widehat{\dk}_{m,n})$, where we view $\dk_{m,n}$ as a Lie subalgebra of $\dk_{m+n}$.
\end{claim}

\begin{proof}[Proof of the claim]
Basically, this is because there is no crossing between the strands in $\beta^0$ which were poles of $\beta$.
More details are as follows.

{\em Step 1.}
Assume that $\beta$ is one of the elements in \eqref{eq:gen_paeb}.
Then, $\beta^0$ is either $\sigma^{\pm}$ or $\alpha$.
If $\beta^0 = \sigma^{\pm}$, then $B = \exp \left( \pm \frac{1}{2} t_{12} \right) = \exp \left( \pm \frac{1}{2} a_{11} \right) \in \exp(\widehat{\dk}_{1,1})$.
If $\beta^0 = \alpha$, there are three possibilities: (i) $\beta = \alpha_{\rm pps}$; (ii) $\beta = \alpha_{\rm psp}$; (iii) $\beta = \alpha_{\rm spp}$.
Since $\varphi$ has no linear term, we have $\varphi(-t_{13}-t_{23}, t_{23}) = \varphi(t_{12}, t_{23}) = \varphi(t_{12}, -t_{12}-t_{13})$.
Hence, we have $B = \Phi(-a_{11}-a_{21},a_{21})$ in case~(i), $B = \Phi(a_{11},a_{21})$ in case~(ii), and $B = \Phi(a_{11},-a_{11}-a_{21})$ in case~(iii).
In all cases, we obtain that $B \in \exp(\widehat{\dk}_{2,1})$.

{\em Step 2.}
By Proposition~\ref{prop:paeb_gen} we can decompose $\beta$ into a product of basic morphisms (in the sense of the proof of Proposition~\ref{prop:paeb_gen}).
By Step 1, any basic morphism is sent by $Z^{\rm pb}$ to a morphism in $\widehat{{\bf PaCD}}$ whose first component lies in $\exp(\widehat{\dk}_{m,n})$.
Hence the same is true for $\beta$, and the claim follows.
\end{proof}

The functor $Z^{\rm pb}$ is filtration-preserving, induces the identity at the associated graded, and respects all the operadic operations in ${\bf PaB}$.
Therefore, the induced functor $Z^{\rm mb}$ satisfies all the required properties for a homomorphic expansion for ${\bf PaMB}$.
This completes the proof of Proposition~\ref{prop:ass}.
\end{proof}

\begin{remark} \label{rem:torsor}
Following Bar-Natan's interpretation \cite{BN98} of the set ${\sf Assoc}_1$ of Drinfeld associators and the Grothendieck-Teichm\"{u}ller groups ${\rm GT}_1$ and ${\rm GRT}_1$, let ${\sf Assoc}_1^{\rm mb}$ be the set of homomorphic expansions for ${\bf PaMB}$, and let ${\rm GT}_1^{\rm mb}$ (resp.\,${\rm GRT}_1^{\rm mb}$) be the group of structure preserving automorphisms of $\widehat{{\bf PaMB}}$ (resp.\,$\widehat{\bf PaMCD}$).
Here, ``structure preserving'' means the same as in Definition~\ref{def:expansion}.
Then Proposition~\ref{prop:ass} shows that there is an (injective) map ${\sf Assoc}_1 \to {\sf Assoc}_1^{\rm mb}$ (implying that ${\sf Assoc}_1^{\rm mb}$ is nonempty), and there is a bitorsor structure 
\[
{\rm GRT}_1^{\rm mb}
\curvearrowright
{\sf Assoc}_1^{\rm mb}
\curvearrowleft
{\rm GT}_1^{\rm mb}
\]
where the actions are given by pre-composition and post-composition.
Using the same argument as the proof of Proposition~\ref{prop:ass}, one can define group homomorphisms ${\rm GT}_1 \to {\rm GT}_1^{\rm mb}$ and ${\rm GRT}_1 \to {\rm GRT}_1^{\rm mb}$, and make the map ${\sf Assoc}_1 \to {\sf Assoc}_1^{\rm mb}$ a morphism of bitorsors.
\end{remark}

Are there any other ways to obtain homomorphic expansions for ${\bf PaMB}$?
In other words using the notation in Remark~\ref{rem:torsor}, is the map ${\sf Assoc}_1 \to {\sf Assoc}_1^{\rm mb}$ surjective?
By Proposition~\ref{prop:paeb_gen}, any $Z^{\rm mb}$ is specified by values on the basic morphisms in \eqref{eq:gen_paeb}.
For the first and second elements, the group-like condition for $Z^{\rm mb}$ implies that
$Z^{\rm mb} ( \sigma_{\rm ps}^+) =
\big( \exp( \lambda a_{11}), \begin{tikzpicture}[baseline=3pt, x=1mm, y=2mm]
\draw[rs] (0,0) -- (2,2);
\draw[bs] (2,0) -- (0,2);
\end{tikzpicture} \big)$ and 
$Z^{\rm mb} (\sigma_{\rm ps}^-) =
\big( \exp( \mu a_{11}), \begin{tikzpicture}[baseline=3pt, x=1mm, y=2mm]
\draw[rs] (0,0) -- (2,2);
\draw[bs] (2,0) -- (0,2);
\end{tikzpicture} \big)$
 for some $\lambda, \mu \in \Q$.
 Applying the operation $\vartheta_1$ to the first equation, we obtain 
 $Z^{\rm mb} \big( 
\begin{tikzpicture}[baseline=3pt, x=2mm, y=2mm]
\pc{bs}{bs}{0,0}
\end{tikzpicture}
\big) = \big( \exp( \lambda c_{12}), \begin{tikzpicture}[baseline=3pt, x=1mm, y=2mm]
\draw[bs] (0,0) -- (2,2);
\draw[bs] (2,0) -- (0,2);
\end{tikzpicture} \big)$
and 
$
Z^{\rm mb}\big( (\tau_{11}, \begin{tikzpicture}[baseline=3pt, x=1mm, y=2mm]
\draw[bs] (0,0) -- (0,2);
\draw[bs] (2,0) -- (2,2);
\end{tikzpicture}) \big)
= Z^{\rm mb} \big( 
\begin{tikzpicture}[baseline=4pt, x=2mm, y=1.25mm]
\pc{bs}{bs}{0,0}
\pc{bs}{bs}{0,2}
\end{tikzpicture} \big)
= \big( \exp( 2 \lambda c_{12}), \begin{tikzpicture}[baseline=3pt, x=1mm, y=2mm]
\draw[bs] (0,0) -- (0,2);
\draw[bs] (2,0) -- (2,2);
\end{tikzpicture} \big)
= \big( 1 + 2 \lambda c_{12}, \begin{tikzpicture}[baseline=3pt, x=1mm, y=2mm]
\draw[bs] (0,0) -- (0,2);
\draw[bs] (2,0) -- (2,2);
\end{tikzpicture} \big)$.
Since ${\rm gr}\, Z^{\rm mb}$ is the identity, we obtain $\lambda = 1/2$.
Similarly, we obtain $\mu = -1/2$.
In summary, we have 
\begin{equation} \label{eq:Z_normalized}
Z^{\rm mb} (\sigma_{\rm ps}^+)
= \left( \exp \left( \frac{1}{2} a_{11} \right),
\begin{tikzpicture}[baseline=8pt, x=1.8mm, y=3.6mm]
\draw[rs] (0,0) -- (2,2);
\draw[bs] (2,0) -- (0,2);
\end{tikzpicture} \right),
\quad 
Z^{\rm mb} (\sigma_{\rm ps}^-)
= \left( \exp \left( -\frac{1}{2} a_{11} \right),
\begin{tikzpicture}[baseline=8pt, x=1.8mm, y=3.6mm]
\draw[rs] (0,0) -- (2,2);
\draw[bs] (2,0) -- (0,2);
\end{tikzpicture} \right).
\end{equation} 
For the other three morphisms in \eqref{eq:gen_paeb}, we write
\[
Z^{\rm mb}(\alpha_{\rm pps})=\left( \Phi_{\rm pps}, 
\begin{tikzpicture}[baseline=8pt, x=0.9mm, y=1.2mm]
\draw[rs] (0,0) -- (0,6);
\draw[rs] (1,0) -- (4,6);
\draw[bs] (5,0) -- (5,6);
\end{tikzpicture} \, 
\right),
\quad
Z^{\rm mb}(\alpha_{\rm psp})=\left( \Phi_{\rm psp},
\begin{tikzpicture}[baseline=8pt, x=0.9mm, y=1.2mm]
\draw[rs] (0,0) -- (0,6);
\draw[bs] (1,0) -- (4,6);
\draw[rs] (5,0) -- (5,6);
\end{tikzpicture} \, 
\right),
\]
\[
Z^{\rm mb}(\alpha_{\rm spp})=\left( \Phi_{\rm spp},
\begin{tikzpicture}[baseline=8pt, x=0.9mm, y=1.2mm]
\draw[bs] (0,0) -- (0,6);
\draw[rs] (1,0) -- (4,6);
\draw[rs] (5,0) -- (5,6);
\end{tikzpicture} \, 
\right).
\]
By the group-like condition for $Z^{\rm mb}$, the first component of these three elements lie in $\exp(\widehat{\dk}_{2,1}) \subset \widehat{\mathcal{A}}_{2,1}$ so that one can write $\Phi_{\rm pps} = \exp(\varphi_{\rm pps})$, $\Phi_{\rm psp} = \exp(\varphi_{\rm psp})$, and $\Phi_{\rm spp} = \exp(\varphi_{\rm spp})$ for some $\varphi_{\rm pps}, \varphi_{\rm psp}, \varphi_{\rm spp} \in \widehat{\dk}_{2,1}$.

To obtain a well-defined functor $Z^{\rm mb}$, the values $\Phi_{\rm pps}$, $\Phi_{\rm psp}$ and $\Phi_{\rm spp}$ have to satisfy several equations coming from equalities among parenthesized mixed braids.
We focus on the following ppss-pentagon equality
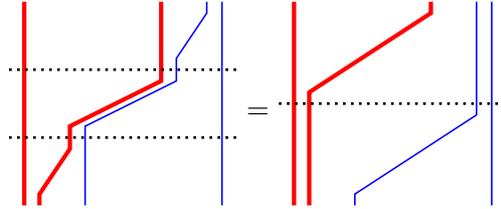
\[
\begin{tikzpicture}[baseline=33pt, x=2mm, y=3mm]
\draw[rs] (0,0) -- (0,9);
\draw[rs] (1,0) -- (1,0.5) -- (3,2.5) -- (3,3) -- (3,3.5) -- (9,5.5) -- (9,6) -- (9,9);
\draw[bs] (4,0) -- (4,3) -- (4,3.5) -- (10,5.5) -- (10,6) -- (10,6.5) -- (12,8.5) -- (12,9);
\draw[bs] (13,0) -- (13,9);
\draw[dotted] (-1,3) -- (14,3);
\draw[dotted] (-1,6) -- (14,6);
\end{tikzpicture}
= 
\begin{tikzpicture}[baseline=33pt, x=2mm, y=3mm]
\draw[rs] (0,0) -- (0,9);
\draw[rs] (1,0) -- (1,4.5) -- (1,5) -- (9,8.5) -- (9,9);
\draw[bs] (4,0) -- (4,0.5) -- (12,4) -- (12,4.5) -- (12,9);
\draw[bs] (13,0) -- (13,9);
\draw[dotted] (-1,4.5) -- (14,4.5);
\end{tikzpicture}
\]
as morphisms from $((\rb \rb)\bb) \bb$ to $\rb (\rb(\bb \bb))$.
It can be written as 
\[
d^4(\alpha_{\rm pps}) d^2(\alpha_{\rm pps}) d^0(\alpha_{\rm pps}) 
= d^1(\alpha_{\rm pps}) d^3(\alpha_{\rm pps}).
\]
Applying $Z^{\rm mb}$, we obtain
\begin{equation} \label{eq:PPSS-5gon}
d^4 (\Phi_{\rm pps}) d^2(\Phi_{\rm pps}) d^0(\Phi_{\rm pps}) = d^1(\Phi_{\rm pps}) d^3(\Phi_{\rm pps}).
\end{equation}
The linearization of this equation, which takes the following form:
\begin{equation} \label{eq:ppps5}
d^4 (\varphi_{\rm pps}) + d^2(\varphi_{\rm pps}) + d^0(\varphi_{\rm pps}) = d^1(\varphi_{\rm pps}) + d^3(\varphi_{\rm pps}),
\end{equation}
namely, $d^{2,1}(\varphi_{\rm pps}) = 0$.
We call this the linearized ppss-pentagon equation.

\begin{remark}
To give a complete description of homomorphic expansions for ${\bf PaMB}$, we need to know a (finite) presentation of the category ${\bf PaMB}$ in terms of the generators in \eqref{eq:gen_paeb}.
An analogous result is well known for ${\bf PaB}$ (see e.g., \cite[\S 3]{BN98}).
We do not pursue this issue in the present paper, but it should be useful for studying ${\sf Assoc}_1^{\rm mb}$, ${\rm GT}_1^{\rm mb}$ and ${\rm GRT}_1^{\rm mb}$.
\end{remark}

\subsection{Emergent pentagon equation} \label{subsec:56gon}

From now on, we consider the case $k=2$, namely ${\bf PaMB}^{/2} = {\bf PaEB}$, and focus on the linearized ppss-pentagon equation \eqref{eq:ppps5} in $\edk_{2,2}$.

\begin{prop} \label{prop:ppss5} 
Let $\varphi = \varphi(x,y) \in \lie_2$.
Then, $\varphi_1 = \varphi(a_{11},a_{21}) \in \edk_{2,1}$ satisfies $d^{2,1}(\varphi_1) = 0 \in \edk_{2,2}$ if and only if $\varphi$ satisfies the following two equations:
\begin{align}
& \varphi(y,0) - \varphi(x+y,0) = 0,  \label{eq:p1} \\
& (\partial_y \varphi)(x,y) + (\partial_y \varphi)(y,0) - (\partial_y \varphi)(x+y,0) - R(\varphi) = 0. \label{eq:ppss5-gon}
\end{align}
\end{prop}

\begin{proof}
Using formulas~\eqref{eq:del0}, \eqref{eq:delk} and Lemmas~\ref{lem:edk_cab_strnd}, \ref{lem:edk_taum}, we obtain that
\begin{align*}
d_0(\varphi_1) &= \varphi(y,0)_2 + (\pa_y \varphi)(y,0)_{12}, \\
d_1(\varphi_1) &= \varphi(x+y,0)_2 + (\pa_y \varphi)(x+y,0)_{12}, \\
d_2(\varphi_1) &= \varphi(x,y)_2 + (\pa_y \varphi)(x,y)_{12}, \\
d_3(\varphi_1) &= \varphi(x,y)_1 + \varphi(x,y)_2 + R(\varphi)_{12}, \\
d_4(\varphi_1) &= \varphi(x,y)_1.
\end{align*}
The assertion follows from this.
\end{proof}

\begin{remark} \label{rem:p1}
Equation~\eqref{eq:p1} says that the coefficient of $x$ in $\varphi$ is zero.
One can check that in degree one, solutions to equation~\eqref{eq:ppss5-gon} are scalar multiples of $x$.
Hence, there is no nontrivial solution to $d^{2,1}(\varphi_1) = 0$ in degree one, and equation~\eqref{eq:p1} is redundant in degrees at least two.
\end{remark}

Recall the definition of the Grothendieck-Teichm\"{u}ller Lie algebra $\grt_1$ by Drinfeld~\cite{Drinfeld}.
It is the space of $\psi \in \lie_2$ which satisfy the following relations:
\begin{align}
& \psi(x,y) = -\psi(y,x), \nonumber \\
& \psi(x,y) + \psi(y,-x-y) + \psi(-x-y,x) = 0, \nonumber \\
& \psi(t_{12},t_{2(34)}) + \psi(t_{(12)3}, t_{34}) = \psi(t_{23}, t_{34}) + \psi(t_{1(23)},t_{(23)4}) + \psi(t_{12},t_{23}). \label{eq:5-gon}
\end{align}
Here, the last equation \eqref{eq:5-gon}, called the pentagon equation, takes place in $\dk_4$ and $t_{2(34)} = t_{23} + t_{24}$, etc.
It is known that nontrivial elements in $\grt_1$ have degrees at least three.

The pentagon equation can be described in terms of a certain differential on $\dk_n$.
In fact, the differential on $\edk_{m,n}$ is induced from this differential.
Assume that $\psi \in \lie_2$ has degree at least two.
Then, it can be considered as an element in $\dk_3$ by the substitution $\psi \mapsto \psi(t_{12}, t_{23})$.
There are maps $d_i: \dk_3 \to \dk_4$ for $0 \le i \le 4$ defined in terms of extension and cabling operations, and $\psi$ is a solution to the pentagon equation if and only if $d^3(\psi) = \sum_{i=0}^4 (-1)^i d_i(\psi) = 0$.
Furthermore, through the isomorphism
\[
\dk_3 \cong \Q (t_{12} + t_{13} + t_{23})
\oplus \lie(t_{13}, t_{23}),
\]
$\psi(t_{12},t_{23})$ corresponds to $\psi(-t_{13}-t_{23},t_{23}) \in \lie(t_{13}, t_{23})$.
Now, identify $\lie(t_{13}, t_{23})$ with $\dk_{2,1} = \lie(a_{11},a_{21}) \cong \edk_{2,1}$ by $t_{13} \mapsto a_{11}$ and $t_{23} \mapsto a_{21}$.
Since the coface maps on $\dk_3$ and $\edk_{2,1}$ are compatible with this identification, it follows that any $\psi = \psi(x,y) \in \lie_2$ with $d^3(\psi) = 0$ satisfies $d^{2,1}(\psi(-x-y,y)) = 0$.

We now want to introduce the emergent version of the Grothendieck-Teichm\"{u}ller Lie algebra as the space of solutions to the linearized ppss-pentagon equation.
For a technical reason, we put an additional condition coming from the following fact proved by  
Drinfeld \cite[equation (5.19)]{Drinfeld}: any $\psi \in \grt_1$ satisfies $[x, \psi(-x-y,x)] + [y, \psi(-x-y,y)] = 0$, and hence $\varphi(x,y) = \psi(-x-y,y)$ satisfies 
\begin{equation} \label{eq:vscondition}
[x, \varphi(y,x)] + [y, \varphi(x,y)] = 0.
\end{equation}

\begin{dfn}
Let 
\[
\grt_1^\EM := \{ \varphi \in \lie_2 \mid 
\text{$\varphi$ satisfies equations~\eqref{eq:p1}, \eqref{eq:ppss5-gon} and \eqref{eq:vscondition}} \}.
\]
\end{dfn}

By definition, we have an injection $\grt_1 \hookrightarrow \grt_1^\EM, \psi(x,y) \mapsto \psi(-x-y,y)$.

\begin{remark}
A computer experiment shows that up to degree $17$, the space of solutions to equations~\eqref{eq:p1} and \eqref{eq:ppss5-gon} coincides with $\grt_1$.
Furusho~\cite{Furusho} showed that the pentagon equation \eqref{eq:5-gon} implies the other two defining equations for $\grt_1$.
We do not know whether an analogous result holds true for $\grt_1^\EM$, namely whether equations~\eqref{eq:p1} and \eqref{eq:ppss5-gon} imply equation~\eqref{eq:vscondition}.
\end{remark}

\begin{remark}
There is a topological explanation for the source of condition~\eqref{eq:vscondition}. 
Consider the following hexagonal equalities: 
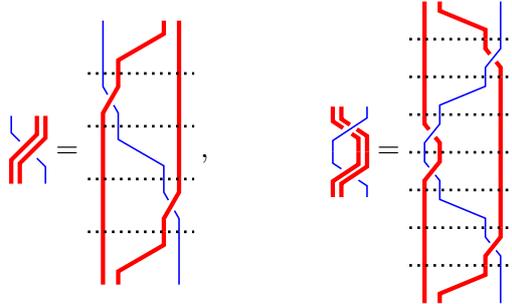
\[
\begin{tikzpicture}[baseline=10pt, x=4.5mm, y=4.5mm]
\draw[rs] (0.25,0) -- (0.25,0.6) -- (1,1.35) -- (1,2);
\draw[rs] (0,0) -- (0,0.7) -- (0.75,1.45) -- (0.75,2);
\draw[bs] (1,0) -- (1,0.5) -- (0.75, 0.75);
\draw[bs] (0,2) -- (0,1.5) -- (0.25,1.25); 
\end{tikzpicture}
=
\begin{tikzpicture}[baseline=48pt, x=2mm, y=3.5mm]
\draw[rs] (0,0) -- (0,6);
\draw[rs] (1,0) -- (1,0.5) -- (4,1.5) -- (4,2);
\draw[bs] (5,0) -- (5,2);
\pc{rs}{bs}{4,2}
\draw[bs] (4,4) -- (4,4.5) -- (1,5.5) -- (1,6);
\pc{rs}{bs}{0,6}
\draw[rs] (5,4) -- (5,10);
\draw[bs] (0,8) -- (0,10);
\draw[rs] (1,8) -- (1,8.5) -- (4,9.5) -- (4,10);
\draw[dotted] (-1,2) -- (6,2);
\draw[dotted] (-1,4) -- (6,4);
\draw[dotted] (-1,6) -- (6,6); 
\draw[dotted] (-1,8) -- (6,8);
\end{tikzpicture} \, ,
\hspace{4em}
\begin{tikzpicture}[baseline=15pt, x=4.5mm, y=3mm]
\draw[rs] (0.25,0) -- (0.25,0.6) -- (1,1.35) -- (1,2);
\draw[rs] (0,0) -- (0,0.7) -- (0.75,1.45) -- (0.75,2);
\draw[bs] (1,0) -- (1,0.5) -- (0.75, 0.75);
\draw[bs] (0,2) -- (0,1.5) -- (0.25,1.25); 
\draw[rs] (0.75,2) -- (0.75,2.6) -- (0.55,2.8);
\draw[rs] (0.3,3.05) -- (0,3.35) -- (0,4);
\draw[rs] (1,2) -- (1,2.7) -- (0.7,3); 
\draw[rs] (0.5,3.2) -- (0.25,3.45) -- (0.25,4);
\draw[bs] (0,2) -- (0,2.5) -- (1,3.5) -- (1,4);
\end{tikzpicture}
= 
\begin{tikzpicture}[baseline=55pt, x=2mm, y=2.5mm]
\draw[rs] (0,0) -- (0,6);
\draw[rs] (1,0) -- (1,0.5) -- (4,1.5) -- (4,2);
\draw[bs] (5,0) -- (5,2);
\pc{rs}{bs}{4,2}
\draw[bs] (4,4) -- (4,4.5) -- (1,5.5) -- (1,6);
\draw[rs] (5,4) -- (5,12);
\pc{rs}{bs}{0,6}
\pc{bs}{rs}{0,8}
\draw[rs] (0,10) -- (0,16);
\draw[bs] (1,10) -- (1,10.5) -- (4,11.5) -- (4,12);
\pc{bs}{rs}{4,12}
\draw[rs] (4,14) -- (4,14.5) -- (1,15.5) -- (1,16);
\draw[bs] (5,14) -- (5,16);
\draw[dotted] (-1,2) -- (6,2);
\draw[dotted] (-1,4) -- (6,4);
\draw[dotted] (-1,6) -- (6,6);
\draw[dotted] (-1,8) -- (6,8);
\draw[dotted] (-1,10) -- (6,10);
\draw[dotted] (-1,12) -- (6,12);
\draw[dotted] (-1,14) -- (6,14);
\end{tikzpicture}
\]
From the first equality and \eqref{eq:Z_normalized}, we obtain
\begin{equation} \label{eq:Phi1}
e^{\frac{x+y}{2}} = \Phi_{\rm pps}\, e^{\frac{y}{2}}\, \Phi_{\rm psp}^{-1}\, e^{\frac{x}{2}}\, \Phi_{\rm spp}.
\end{equation}
This shows a relationship among the elements $\Phi_{\rm pps}$, $\Phi_{\rm psp}$ and $\Phi_{\rm spp}$.
From the second equality, we obtain
\begin{equation} \label{eq:Phi2}
e^{x+y} = \Phi_{\rm pps}\, e^{\frac{y}{2}}\, \Phi_{\rm psp}^{-1}\, e^x\, \Phi_{\rm psp}\, e^{\frac{y}{2}}\, \Phi_{\rm pps}^{-1}.
\end{equation}
Now, we apply the reflection with respect to a vertical axis to the three $\alpha$'s in \eqref{eq:gen_paeb}.
Then $\alpha_{\rm spp}$ is mapped to the inverse of $\alpha_{\rm pps}$ and $\alpha_{\rm psp}$ to its inverse, where the order of red strands gets reversed.
From this observation, let us consider the following condition:
\begin{equation} \label{eq:hypothetical}
\Phi_{\rm spp}(x,y) = \Phi_{\rm pps}(y,x)^{-1}, 
\qquad
\Phi_{\rm psp}(x,y) = \Phi_{\rm psp}(y,x)^{-1}. 
\end{equation}
We do not know if any homomorphic expansion $Z^{\rm mb}$ satisfies these conditions.
Under this hypothesis, the element $\Phi_{\rm pps}$ and equation~\eqref{eq:Phi1} determine the other two elements $\Phi_{\rm psp}$ and $\Phi_{\rm spp}$.
Moreover, equations~\eqref{eq:Phi1}, \eqref{eq:Phi2} and \eqref{eq:hypothetical} imply that
\[
e^{x+y} = e^{\frac{x+y}{2}}\, \Phi_{\rm pps}(y,x)\, e^x\, \Phi_{\rm pps}(y,x)^{-1}\, e^{-\frac{x+y}{2}} \, \Phi_{\rm pps}\, e^y\, \Phi_{\rm pps}^{-1}.
\]
Taking the linearization of this equation yields condition~\eqref{eq:vscondition}.
\end{remark}

\section{Loop operations and Kashiwara-Vergne theory} \label{sec:lopKV}

In this section, we explain an interpretation of the Kashiwara-Vergne Lie algebras in terms of surface topology \cite{AKKN18, AKKN_hg}.

Fix a positive integer $n$ and let $\Sigma$ be an $n$-punctured disk, that is, a closed unit disk in $\mathbb{R}^2$ with $n$ distinct points in the interior removed.
Choose a basepoint $*$ in the boundary of $\Sigma$ and let $\pi = \pi_1(\Sigma, *)$.
Since the group $\pi$ is free of rank $n$, the associated graded quotient of the group algebra $\Q \pi$ with respect to the powers of the augmentation ideal is canonically isomorphic to the free associative algebra $\ass_n = \ass(x_1, \ldots, x_n)$, where the generator $x_i$ corresponds to the homology class of the loop around the $i$th puncture:
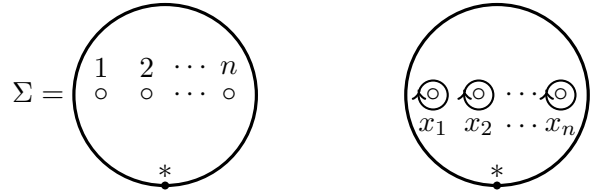
\[
\Sigma = 
\begin{tikzpicture}[baseline=-2.5pt, x=2.4mm, y=2.4mm]
\draw[line width=1.2pt] (0,0) circle[radius=5];
\draw (0,-4.2) node{$*$};
\fill (0,-5) circle[radius=1.5pt];
\draw (-3.5,0) node{$\circ$};
\draw (-1,0) node{$\circ$};
\draw (1.5,0) node{$\cdots$};
\draw (3.5,0) node{$\circ$};
\draw (-3.5,1.5) node{$1$};
\draw (-1,1.5) node{$2$};
\draw (1.5,1.5) node{$\cdots$};
\draw (3.5,1.5) node{$n$};
\end{tikzpicture}
\hspace{5em}
\begin{tikzpicture}[baseline=-2.5pt, x=2.4mm, y=2.4mm]
\draw[line width=1.2pt] (0,0) circle[radius=5];
\draw (0,-4.2) node{$*$};
\fill (0,-5) circle[radius=1.5pt];
\draw (-3.5,0) node{$\circ$};
\draw (-1,0) node{$\circ$};
\draw (1.5,0) node{$\cdots$};
\draw (3.5,0) node{$\circ$};
\draw[line width=0.8pt] (-3.5,0) circle[radius=0.8];
\draw[line width=0.8pt] (-1,0) circle[radius=0.8];
\draw[line width=0.8pt] (3.5,0) circle[radius=0.8];
\draw[->] (-4.275,0.17) -- (-4.275,0.17);
\draw[->] (-1.775,0.17) -- (-1.775,0.17);
\draw[->] (2.725,0.17) -- (2.725,0.17);
\draw (-3.5,-1.75) node{$x_1$} (-1,-1.75) node{$x_2$} (1.5,-1.75) node{$\cdots$} (3.5,-1.75) node{$x_n$};
\end{tikzpicture}
\]

Let $A$ be an associative $\Q$-algebra.
The trace space $|A|$ is defined to be $A/[A,A]$.
We denote by $|\ |: A \to |A|$ the natural projection.
For instance, the space $|\Q\pi|$ is naturally identified with the set of homotopy classes of free loops in $\Sigma$, and $\tr_n:= |\ass_n|$ is the space of cyclic words in $x_1, \ldots, x_n$.

\subsection{Loop operations on a punctured disk} \label{subsec:loop_ops}

We briefly recall several loop operations on $\Q \pi$ and $|\Q \pi|$.
Our focus is on their linearized version, namely the associated graded operations on $\ass_n$ and $\tr_n$.
For more details about the loop operations themselves, see \cite{MT13, Mas18, AKKN_hg}.

The main cast in the sequel are the (associated graded operations of) the homotopy intersection form \cite{MT13} and the framed Turaev cobracket \cite{AKKN18, AKKN_hg}.
The former is a $\Q$-linear map 
$
\eta : \Q \pi^{\otimes 2} \to \Q\pi
$
defined in terms of intersections of two based loops in $\Sigma$, and the latter is a $\Q$-linear map
$
\delta^f : |\Q \pi| \to |\Q \pi|^{\otimes 2}
$
defined in terms of self-intersections of a free loop in $\Sigma$.
The operation $\delta^f$ depends on the choice of a framing on $\Sigma$.
Here, we choose the blackboard framing associated with the inclusion $\Sigma \subset \mathbb{R}^2$.
We give sample computations of these operations:
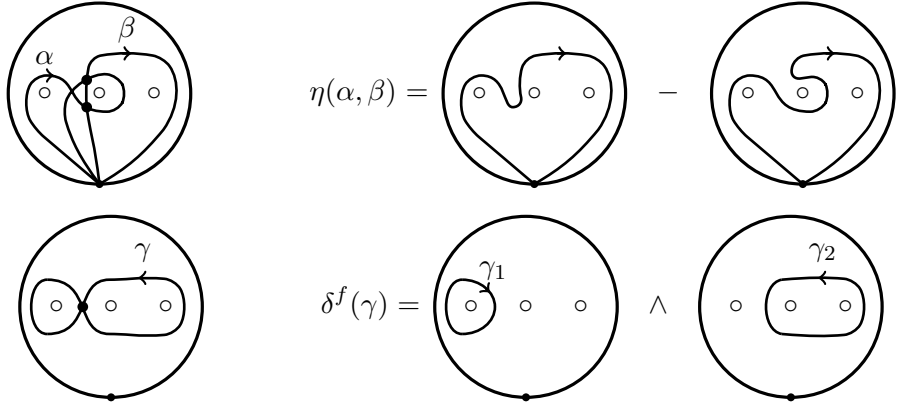
\[
\begin{tikzpicture}[baseline=-2.5pt, x=2.4mm, y=2.4mm]
\draw[line width=1.2pt] (0,0) circle[radius=5];
\fill (0,-5) circle[radius=1.5pt];
\draw (-3,0) node{$\circ$};
\draw (0,0) node{$\circ$};
\draw (3,0) node{$\circ$};
\draw (0,-5) {[rounded corners=4pt] to (-3.5,-2) to[bend left=30] (-3.5,1) ..controls(-1,1)and(-2,-1).. (1,-1) to (1.5,0) to (1,1) to[bend right=40] (-2,-1) to[bend right=10] (0,-5)};
\draw[->] (-2.75,1) -- (-2.5,1);
\draw (-3,2) node{$\alpha$};
\draw (0,-5) {[rounded corners=4pt] to[bend right=2.5] (-0.5,-2) to[bend left=10] (-0.5,2) to[bend left=10] (3.5,2)} to[bend left=35] (3.5,-2) {[rounded corners=4pt] to[bend left=10] (0,-5)};
\draw[->] (1.25,2.22) -- (1.75,2.22);
\draw (1.5,3.5) node{$\beta$};
\fill (-0.7, -0.75) circle[radius=2pt];
\fill (-0.7, 0.75) circle[radius=2pt];
\end{tikzpicture}
\hspace{4em}
\eta(\alpha, \beta) =\, 
\begin{tikzpicture}[baseline=-2.5pt, x=2.4mm, y=2.4mm]
\draw[line width=1.2pt] (0,0) circle[radius=5];
\fill (0,-5) circle[radius=1.5pt];
\draw (-3,0) node{$\circ$};
\draw (0,0) node{$\circ$};
\draw (3,0) node{$\circ$};
\draw (0,-5) {[rounded corners=4pt] to (-3.5,-2) to[bend left=30] (-3.5,1) ..controls(-1.5,1)and(-1.75,-0.75).. (-0.6,-0.75) to[bend left=15] (-0.5,2) to[bend left=10] (3.5,2)} to[bend left=35] (3.5,-2) {[rounded corners=4pt] to[bend left=10] (0,-5)};
\draw[->] (1.25,2.22) -- (1.75,2.22);
\end{tikzpicture}
\hspace{0.8em} - \hspace{0.8em} 
\begin{tikzpicture}[baseline=-2.5pt, x=2.4mm, y=2.4mm]
\draw[line width=1.2pt] (0,0) circle[radius=5];
\fill (0,-5) circle[radius=1.5pt];
\draw (-3,0) node{$\circ$};
\draw (0,0) node{$\circ$};
\draw (3,0) node{$\circ$};
\draw (0,-5) {[rounded corners=4pt] to (-3.5,-2) to[bend left=30] (-3.5,1) ..controls(-1,1)and(-2,-1).. (1,-1) to (1.5,0) to (1,1) to (-0.7,0.75) to (-0.5,2) to[bend left=10] (3.5,2)} to[bend left=35] (3.5,-2) {[rounded corners=4pt] to[bend left=10] (0,-5)};
\draw[->] (1.25,2.22) -- (1.75,2.22);
\end{tikzpicture}
\]
\[
\begin{tikzpicture}[baseline=-2.5pt, x=2.4mm, y=2.4mm]
\draw[line width=1.2pt] (0,0) circle[radius=5];
\fill (0,-5) circle[radius=1.5pt];
\draw (-3,0) node{$\circ$};
\draw (0,0) node{$\circ$};
\draw (3,0) node{$\circ$};
\draw[rounded corners=4pt] (-3.5,1.5) ..controls(-1,1.5)and(-2,-1.5).. (0.5,-1.5) to[bend right=5] (3.5,-1.5) to[bend right=30] (3.5,1.5) to[bend right=5] (0.5,1.5) ..controls(-2,1.5)and(-1,-1.5).. (-3.5,-1.5);
\draw (-3.5,-1.5) to[bend left=90] (-3.5,1.5);
\draw[->] (1.75,1.6) -- (1.5,1.6);
\draw (1.75,3) node{$\gamma$};
\fill (-1.55,0) circle[radius=2pt];
\end{tikzpicture}
\hspace{4em}
\delta^f(\gamma) = \,
\begin{tikzpicture}[baseline=-2.5pt, x=2.4mm, y=2.4mm]
\draw[line width=1.2pt] (0,0) circle[radius=5];
\fill (0,-5) circle[radius=1.5pt];
\draw (-3,0) node{$\circ$};
\draw (0,0) node{$\circ$};
\draw (3,0) node{$\circ$};
\draw (-3.5,1.5) arc[x radius=1.8, y radius=1.5, start angle=90, end angle=-90];
\draw (-3.5,-1.5) to[bend left=90] (-3.5,1.5);
\draw[->] (-2.2,1) -- (-2,0.8);
\draw (-1.8,2) node{$\gamma_1$};
\end{tikzpicture}
\hspace{0.8em} \wedge \hspace{0.8em} 
\begin{tikzpicture}[baseline=-2.5pt, x=2.4mm, y=2.4mm]
\draw[line width=1.2pt] (0,0) circle[radius=5];
\fill (0,-5) circle[radius=1.5pt];
\draw (-3,0) node{$\circ$};
\draw (0,0) node{$\circ$};
\draw (3,0) node{$\circ$};
\draw[rounded corners=4pt] (-0.5,-1.5) to[bend right=5] (3.5,-1.5) to[bend right=30] (3.5,1.5) to[bend right=5] (-0.5,1.5);
\draw[->] (1.75,1.6) -- (1.5,1.6);
\draw (-0.5,-1.5) to[bend left=80] (-0.5,1.5);
\draw (1.75,3) node{$\gamma_2$};
\end{tikzpicture}
\]
Here, $\gamma_1 \wedge \gamma_2 = \gamma_1 \otimes \gamma_2 - \gamma_2 \otimes \gamma_1$.
In the first example, there are two intersections of $\alpha$ and $\beta$, and each of them contributes to a term in $\eta(\alpha, \beta)$. The sign is determined by the local index at each intersection. 
In the second example, $\gamma$ has one self-intersection.
In general, $\delta^f(\gamma)$ is obtained by splitting $\gamma$ into two free loops at each intersection.
We also use a based loop version of $\delta^f$, namely a certain map $\mu^f_r : \Q \pi \to |\Q \pi| \otimes \Q \pi$ introduced in \cite[Section~2.3]{AKKN_hg}.
We set $\mu^f:= (\varepsilon \otimes {\rm id}) \circ \mu^f_r: \Q \pi \to \Q \pi$, where the map $\varepsilon: |\Q \pi| \to \Q$ is induced from the augmentation map of $\Q \pi$.
Here is a sample computation of $\mu^f$:
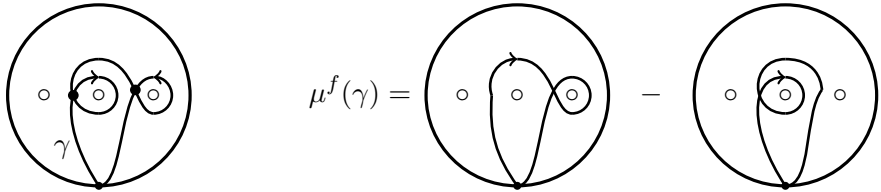
\[
\begin{tikzpicture}[baseline=-2.5pt, x=2.4mm, y=2.4mm]
\draw[line width=1.2pt] (0,0) circle[radius=5];
\fill (0,-5) circle[radius=1.5pt];
\draw (-3,0) node{$\circ$};
\draw (0,0) node{$\circ$};
\draw (3,0) node{$\circ$};
\draw[->] (0,-5) ..controls(-2,-2)and(-2,1).. (0,1); 
\draw (0,-1) arc[x radius=1, y radius=1, start angle=-90, end angle=90];
\draw (0,-1) ..controls(-2,-1)and(-2,2).. (0,2);
\draw (0,2) ..controls(2,2)and(2,-1).. (3,-1); 
\draw[->] (3,-1) arc[x radius=1, y radius=1, start angle=-90, end angle=90];
\draw (3,1) ..controls(1,1)and(1.5,-5).. (0,-5);
\draw (-2,-3) node{$\gamma$}; 
\fill (-1.4,0) circle[radius=2pt];
\fill (2,0.3) circle[radius=2pt];
\end{tikzpicture}
\hspace{4em}
\mu^f(\gamma) =\, 
\begin{tikzpicture}[baseline=-2.5pt, x=2.4mm, y=2.4mm]
\draw[line width=1.2pt] (0,0) circle[radius=5];
\fill (0,-5) circle[radius=1.5pt];
\draw (-3,0) node{$\circ$};
\draw (0,0) node{$\circ$};
\draw (3,0) node{$\circ$};
\draw (0,-5) to[bend left=20] (-1.4,0);
\draw[->] (-1.4,0) ..controls(-1.75,1)and(-1,2).. (0,2);
\draw (0,2) ..controls(2,2)and(2,-1).. (3,-1); 
\draw (3,-1) arc[x radius=1, y radius=1, start angle=-90, end angle=90];
\draw (3,1) ..controls(1,1)and(1.5,-5).. (0,-5);
\end{tikzpicture}
\hspace{0.8em} - \hspace{0.8em} 
\begin{tikzpicture}[baseline=-2.5pt, x=2.4mm, y=2.4mm]
\draw[line width=1.2pt] (0,0) circle[radius=5];
\fill (0,-5) circle[radius=1.5pt];
\draw (-3,0) node{$\circ$};
\draw (0,0) node{$\circ$};
\draw (3,0) node{$\circ$};
\draw[->] (0,-5) ..controls(-2,-2)and(-2,1).. (0,1); 
\draw (0,-1) ..controls(-2,-1)and(-2,2).. (0,2);
\draw (0,2) ..controls(2,2)and(2,0.3).. (2,0.3); 
\draw (0,-1) arc[x radius=1, y radius=1, start angle=-90, end angle=90];
\draw (2,0.3) ..controls(1,-1)and(1.5,-5).. (0,-5);
\end{tikzpicture}
\]
There are two self-intersections of $\gamma$, and each of them contributes to a term in $\mu^f(\gamma)$.
The maps $\mu^f$ and $\eta$ are related by the following formula: for any $a,b \in \Q \pi$, 
\begin{equation} \label{eq:mu_product}
\mu^f(ab) = a \mu^f(b) + \mu^f(a) b + \eta(a,b).
\end{equation}
In fact, the operation $\mu^f$ recovers $\mu^f_r$ and $\delta^f$.
The map $\mu^f_r$ coincides with the following composition 
\begin{align}
\Q \pi \xrightarrow{\Delta} & \Q \pi \otimes \Q \pi \xrightarrow{{\rm id}\otimes \mu^f}
\Q \pi \otimes \Q \pi \nonumber \\
& \xrightarrow{{\rm id} \otimes ((\iota \otimes {\rm id}) \circ \Delta)} \Q \pi \otimes \Q \pi \otimes \Q \pi
\xrightarrow{(|\ |\circ {\rm mult})\otimes {\rm id}} |\Q \pi| \otimes \Q \pi. \label{eq:mufr}
\end{align}
Here, $\Delta$ and $\iota$ are the coproduct and antipode on $\Q \pi$ defined by $\Delta(\gamma) = \gamma \otimes \gamma$ and $\iota(\gamma) = \gamma^{-1}$ for $\gamma \in \pi$, and in the last step we use the multiplication map in the algebra $\Q \pi$. 
Furthermore, for any $a \in \Q\pi$ we have 
\begin{equation} \label{eq:delta_mu}
\delta^f ( |a| ) = {\rm Alt}\circ ({\rm id} \otimes |\ |) \circ \mu^f_r(a) + |a| \wedge {\bf 1}.
\end{equation}
Here, ${\rm Alt}(a \otimes b) = a \otimes b - b \otimes a$ and ${\bf 1}$ is the class of the unit in $\Q\pi$.

\begin{remark}
\begin{enumerate}
\item[(i)]
We give several comments about proofs of the formulas above.
First, one can derive formula~\eqref{eq:mu_product} by applying $(\varepsilon \otimes {\rm id})$ to the first equation in \cite[Proposition 2.9 (i)]{AKKN_hg}. A formula similar to \eqref{eq:mu_product} was proved in \cite[(3.3)]{Mas18} for a variant of the map $\mu^f$.
Second, the decomposition \eqref{eq:mufr} of the map $\mu^f_r$ follows directly from the defining formula of $\mu^f_r$. See \cite[Section~2.3, formula~(13)]{AKKN_hg}.
Finally, formula~\eqref{eq:delta_mu} can be found in \cite[Proposition~2.9 (ii)]{AKKN_hg}.

\item[(ii)]
The map $\delta^f$ is a refinement of the Turaev cobracket \cite{Turaev91}, which is a Lie cobracket on the quotient space $|\Q \pi |/\Q {\bf 1}$.
Turaev \cite{Turaev78} also introduced essentially the same operations as $\eta$ and (an unframed version of) $\mu^f$. 
\end{enumerate}
\end{remark}

\subsection{The associated graded operations} \label{subsec:ass_gr_loop_ops}

All the loop operations in the previous section descend to the associated graded operations on $\ass_n$ and $\tr_n$.
We review their explicit formulas.
For more details, see \cite[Section~3]{AKKN_hg}.
The associated graded operation of $\eta$,  
\[
\eta_{\gr} : {\ass_n}^{\otimes 2} \to \ass_n,
\]
is a map of degree $-1$ and given by $\eta_{\gr}(1, v) = \eta_{\gr}(u,1) = 0$ and 
\begin{equation} \label{eq:etagr}
\eta_{\gr}(a_1 \cdots a_l, b_1 \cdots b_m) =
-a_1 \cdots a_{l-1} \mathfrak{z}(a_l,b_1) b_2 \cdots b_m,
\end{equation}
where $l, m \ge 1$, the elements $a_1, \ldots, a_l, b_1, \ldots, b_m$ are of degree $1$, and $\mathfrak{z}$ is defined by $\mathfrak{z}(x_i, x_j) = \delta_{ij} x_i$.
The associated graded operation of $\mu^f$, 
\[
\mu^f_{\gr}: \ass_n \to \ass_n, 
\]
is a map of degree $-1$ and given by the formula
\begin{equation} \label{eq:mufgr}
\mu^f_{\gr}(a_1 \cdots a_m) = 
-\sum_{j=1}^{m-1} a_1 \cdots a_{j-1} \mathfrak{z}(a_j, a_{j+1}) a_{j+2} \cdots a_m,
\end{equation}
where $a_1,\ldots, a_m$ are elements of degree $1$.
The associated graded version of the relations \eqref{eq:mu_product}, \eqref{eq:mufr} and \eqref{eq:delta_mu} holds true.
First, for any $a,b \in \ass_n$
\begin{equation} \label{eq:mu_product_gr}
\mu^f_{\gr}(ab) = a \mu^f_{\gr}(b) + \mu^f_{\gr}(a) b + \eta_{\gr}(a,b).
\end{equation}
Of course, one can directly check this from formulas \eqref{eq:etagr} and \eqref{eq:mufgr}.
Second, the associated graded operation $\mu^f_{r, \gr}$ decomposes as 
\begin{align}
\ass_n \xrightarrow{\Delta} & \ass_n \otimes \ass_n \xrightarrow{{\rm id}\otimes \mu^f_\gr}
\ass_n \otimes \ass_n \nonumber \\
& \xrightarrow{{\rm id} \otimes ((\iota \otimes {\rm id}) \circ \Delta)} \ass_n \otimes \ass_n \otimes \ass_n
\xrightarrow{(|\ |\circ {\rm mult})\otimes {\rm id}} \tr_n \otimes \ass_n. \label{eq:mufr_gr}
\end{align}
Conversely, we have $\mu^f_\gr = (\varepsilon \otimes {\rm id}) \circ \mu^f_{r, \gr}$.
Finally, for any $a \in \ass_n$ 
\begin{equation} \label{eq:delta_mu_gr}
\delta^f_\gr(|a|) = {\rm Alt} \circ ({\rm id} \otimes |\ |) \circ \mu^f_{r, \gr}(a).
\end{equation} 
Note that the term $|a| \wedge {\bf 1}$ in \eqref{eq:delta_mu} does not contribute to the associated graded operation, since it has filtration degree zero. 

\begin{lem} \label{lem:etauv}
For any $a, b \in \lie_n$, we have
$\eta_{\gr}(a,b) = -\sum_{i=1}^n (\pa_i a) x_i \iota(\pa_i b)$.
\end{lem}

\begin{proof}
This follows from $a = \sum_{i=1}^n (\pa_i a) x_i$ and 
$b = \sum_{i=1}^n x_i \iota(\pa_i b)$.
\end{proof}

We show that the map $\mu^f_{\rm gr}$ is related to the map $R$ introduced in Section~\ref{subsec:dooedk}.
This will be a key point for proving Theorem~\ref{thm:main}.

\begin{prop} \label{prop:Rmugr}
The map $R$ coincides with the restriction of $\mu^f_{\rm gr}$ to $\lie_n$.
\end{prop}

\begin{proof}
We have $\mu^f_\gr(x_i) = 0$ for $i=1,\ldots,n$.
Let $a, b \in \lie_n$. 
Equation~\eqref{eq:mu_product_gr} and Lemma~\ref{lem:etauv} shows that $\mu^f_\gr([a,b]) = \mu^f_\gr(ab) - \mu^f_\gr(ba)$ is equal to 
\[
[\mu^f_{\gr}(a), b] + [a, \mu^f_\gr(b)] 
+ \sum_{i=1}^n \left( (\pa_i b) x_i \iota(\pa_i a) - (\pa_i a) x_i \iota( \pa_i b) \right).
\]
Therefore, the map $\mu^f_\gr$ restricted to $\lie_n$ satisfies the same recursive formula as the map $R$.
This proves the proposition.
\end{proof}

\begin{remark}
In \cite[\S 4.3]{Mas18}, Massuyeau gave a $3$-dimensional formula for $\mu^f$ which involves the cabling operation for pure braids on a punctured disk.
It would be interesting to compare his formula with Proposition~\ref{prop:Rmugr}.
\end{remark}

\subsection{Kashiwara-Vergne Lie algebras} \label{subsec:KVliealg}

We recall the definition of the Kashiwara-Vergne Lie algebras \cite{AT12, AKKN18}.

We begin with some preliminary materials.
Let $\tder_n = {\lie_n}^{\oplus n}$.
The grading on $\lie_n$ makes $\tder_n$ a graded $\Q$-vector space.
For $\tilde{u} = (u_1, \ldots, u_n) \in \tder_n$, let $\rho(\tilde{u})$ be a derivation on $\lie_n$ defined by $\rho(\tilde{u})(x_i) = [x_i, u_i]$ for $i = 1, \ldots, n$.
The space $\tder_n$ has a structure of graded Lie algebra whose Lie bracket is given by $[\tilde{u}, \tilde{v}] = \tilde{w} = (w_1,\ldots, w_n)$ with $w_i = [u_i, v_i] + \rho(\tilde{u})(v_i) - \rho(\tilde{v})(u_i)$ for $i = 1, \ldots, n$, and the map $\tilde{u} \mapsto \rho(\tilde{u})$ is a Lie algebra homomorphism to the derivation Lie algebra of $\lie_n$.
Through this homomorphism, $\tder_n$ acts on $\lie_n$, $\ass_n$, $\tr_n$ and their tensor products.
Elements of $\tder_n$ are called {\it tangential derivations}. 
The space $\sder_n$ of {\it special derivations} is defined to be the set of $\tilde{u} \in \tder_n$ annihilating the element $x_0 = \sum_{i=1}^n x_i$, i.e., $\rho(\tilde{u})(x_0) = 0$.
It forms a Lie subalgebra of $\tder_n$.
The {\it divergence cocycle} \cite{AT12} is a Lie $1$-cocycle defined by the following formula:
\[
{\rm div}: \tder_n \to \tr_n, 
\quad \tilde{u} \mapsto \sum_{i=1}^n | x_i (\pa_i u_i)|.
\]

\begin{dfn}
\begin{enumerate}
\item[(i)]
The {\it Kashiwara-Vergne Lie algebra} $\krv_n$ is the space consisting of $\tilde{u} \in \sder_n$ such that 
$
{\rm div}(\tilde{u}) = 
\sum_{i=0}^n |f_i(x_i) |
$
for some formal power series $f_0(s), f_1(s), \ldots, f_n(s) \in \Q[[s]]$.
\item[(ii)]
Let $\krv_n^0$ be the space of $\tilde{u} \in \sder_n$ such that ${\rm div}(\tilde{u}) \in \bigoplus_{i=1}^n \Q |x_i|$.
\end{enumerate}
\end{dfn}

In the definition of $\krv_n$, the functions $f_i(s)$ actually agree with each other modulo the linear part \cite[Proposition~8.5]{AKKN18}.
In particular, if $n = 2$ and $\tilde{u} \in \krv_2$ is of degree $\ge 3$, then there exists an $f(s) \in \Q[[s]]_{\ge 2}$ such that
\begin{equation} \label{eq:krv_2}
{\rm div}(\tilde{u}) = | f(x_1) + f(x_2) - f(x_1 + x_2) |.
\end{equation}

We have the following sequence of inclusions of graded Lie algebras:
\[
\tder_n \supset \sder_n \supset \krv_n \supset \krv_n^0.
\]
The Lie algebras $\sder_n$, $\krv_n$ and $\krv_n^0$ have the following characterizations in terms of (the associated graded of) the loop operations.

\begin{prop} \label{prop:eta_sder}
Let $\tilde{u} = (u_1, \ldots, u_n) \in {\sf tder}_n$.
Then, the following three conditions are equivalent:
\begin{enumerate}
\item[{\rm (i)}] $\tilde{u} \in \sder_n$;
\item[{\rm (ii)}] $\pa_j u_i = \pa^i u_j$ for any $i,j \in \{ 1, \ldots, n\}$;
\item[{\rm (iii)}] $\rho(\tilde{u})$ commutes with $\eta_\gr$, i.e., 
$\rho(\tilde{u}) \circ \eta_{\rm gr}= \eta_{\rm gr} \circ (\rho(\tilde{u}) \otimes {\rm id} + {\rm id} \otimes \rho(\tilde{u}))$.
\end{enumerate}
\end{prop}

\begin{proof}
The following computation proves the equivalence ${\rm (i)} \Leftrightarrow {\rm (ii)}$:
\begin{align*}
\rho(\tilde{u})(x_0) &= 
\sum_{i=1}^n [x_i, u_i] = 
\sum_{i=1}^n x_i u_i - \sum_{j=1}^n u_j x_j 
= \sum_{i,j = 1}^n \big( x_i (\pa_j u_i) x_j 
- x_i (\pa^i u_j) x_j \big).
\end{align*}

To prove the equivalence ${\rm (ii)} \Leftrightarrow {\rm (iii)}$, note that the map $\eta_\gr$ is a Fox pairing \cite{MT13}.
This means that $\eta_\gr$ satisfies
\[
\begin{cases}
\eta_\gr(ab, c) = a \eta_\gr(b,c) + \varepsilon(b) \eta_\gr(a,c), \\ 
\eta_\gr(a, bc) = \eta_\gr(a,b) c + \varepsilon(b) \eta_\gr(a,c)
\end{cases}
\]
for any $a,b,c \in \ass_n$. 
Thanks to this property, the condition (iii) is equivalent to the commutativity of $\rho(\tilde{u})$ and $\eta_\gr$ on generators of $\ass_n$, namely 
\begin{enumerate}
\item[${\rm (iii)}'$]
$u (\eta_\gr(x_i, x_j)) = \eta_\gr(u(x_i), x_j) + \eta_\gr(x_i, u(x_j))
\quad \text{for any $i,j \in \{ 1, \ldots, n\}$}$.
\end{enumerate}
Now we compute 
$
u (\eta_\gr(x_i, x_j) ) = 
u(\mathfrak{z}(x_i, x_j)) = \delta_{ij} u(x_i) = \delta_{ij} [x_i, u_i]
$ and 
\begin{align*}
\eta_\gr(u(x_i), x_j) + \eta_\gr(x_i, u(x_j))
&= \eta_\gr([x_i, u_i], x_j) + \eta_\gr(x_i, [x_j, u_j]) \\ 
&= x_i (\pa_j u_i) x_j - u_i \mathfrak{z}(x_i, x_j) \\
& \qquad + \mathfrak{z}(x_i, x_j) u_j - x_i (\pa^i u_j) x_j \\
&= x_i (\pa_j u_i - \pa^i u_j) x_j + \delta_{ij} [x_i, u_i].
\end{align*}
Hence the condition ${\rm (iii)}'$ is equivalent to (ii).
This completes the proof.
\end{proof}

\begin{remark}
The equivalence ${\rm (i)} \Leftrightarrow {\rm (iii)}$ in Proposition~\ref{prop:eta_sder} is a special case of (the infinitesimal version of) more general results \cite[Lemmas 6.2 and 6.3]{MT13}, \cite[Theorem 2.31]{Naef20}.
\end{remark}

\begin{thm} \label{thm:sderkrv}
Let $\tilde{u} = (u_1, \ldots, u_n) \in {\sf tder}_n$.
\begin{enumerate}
\item[{\rm (i)}]
$\tilde{u} \in {\sf krv}_n \iff
\text{$\rho(\tilde{u})$ commutes with $\eta_{\rm gr}$ and $\delta^f_{\rm gr}$}$.

\item[{\rm (ii)}]
$\tilde{u} \in {\sf krv}_n^0 \iff 
\text{$\rho(\tilde{u})$ commutes with $\eta_{\rm gr}$ and $\mu^f_{\rm gr}$}$.
\end{enumerate}
\end{thm}

\begin{proof}
First note that in (ii) one can replace $\mu^f_\gr$ with $\mu^f_{r, \gr}$, since $\mu^f_{r, \gr}$ is recovered from $\mu^f_\gr$ and vice versa.

In \cite[Theorem~8.21]{AKKN_hg}, it was shown that the Kashiwara-Vergne groups ${\rm KRV}_n$ (resp. ${\rm KRV}_n^0$) is isomorphic to the group of tangential automorphisms of (the completion of) $\ass_n$ that commute with the operations $\eta_\gr$ and $\delta^f_\gr$ (resp. $\eta_\gr$ and $\mu^f_{r, {\rm gr}}$).
As $\krv_n$ and $\krv_n^0$ are the Lie algebras of ${\rm KRV}_n$ and ${\rm KRV}_n^0$, respectively, the assertions (i) and (ii) follow from this result. 
\end{proof}

\begin{remark} \label{rem:review_AKKN}
As the proof shows, Theorem~\ref{thm:sderkrv} is essentially the linearized version of \cite[Theorem~8.21]{AKKN_hg}.
The key ingredients of the proof of the latter are the following facts (see \cite{AKKN18, AKKN_hg}):
\begin{enumerate}
\item[(i)] The space $\tr_n$ has a Lie algebra structure.
It arises as the associated graded operation of the Goldman bracket \cite{Go86} defined on $|\Q \pi|$.

\item[(ii)] There is a surjective Lie algebra homomorphism $\sigma_{\rm gr}: \tr_n \to {\sf sDer}_n$, where ${\sf sDer}_n$ is the free associative version of $\sder_n$, namely the space of $\tilde{u} = (u_1,\ldots,u_n) \in {\ass_n}^{\oplus n}$ such that $\rho(\tilde{u})(x_0) = \sum_{i=1}^n [x_i, u_i] = 0$.
One can naturally regard $\sder_n$ as a Lie subalgebra of ${\sf sDer}_n$.
It holds that $\rho(\sigma_{\rm gr}(a))(b) = [a,b]$ for any $a, b \in \tr_n$.

\item[(iii)]
There is a free associative version of the divergence cocycle, ${\rm Div}: {\sf sDer}_n \to {\tr_n}^{\otimes 2}$.
It satisfies ${\rm Div}(\tilde{u}) = \widetilde{\Delta} ({\rm div}(\tilde{u}))$ for any $\tilde{u} \in \sder_n$.
Here, $\widetilde{\Delta}: \tr_n \to {\tr_n}^{\otimes 2}$ is the injective map that is induced from $({\rm id} \otimes \iota) \circ \Delta: \ass_n \to {\ass_n}^{\otimes 2}$.

\item[(iv)]
The linearized Turaev cobracket and the divergence cocycle are related by the following formula:
\[
\delta^f_{\rm gr} = {\rm Div} \circ \sigma_{\rm gr}.
\]

\item[(v)]
The center of the Lie algebra $\tr_n$ is spanned by $|\Q [[x_i]]|$, $0 \le i \le n$ (see \cite[Theorem 4.15]{AKKN18} and \cite[Theorem 3.18]{AKKN_hg}).
\end{enumerate}
For the convenience of the reader, let us illustrate how these facts lead directly to the first statement of Theorem~\ref{thm:sderkrv}.
We need more materials from \cite{AKKN18, AKKN_hg} for the second statement, and so we omit it.

\begin{proof}[A direct proof of Theorem \ref{thm:sderkrv}~$(i)$]
By Proposition~\ref{prop:eta_sder}, it is sufficient to prove that for any $\tilde{u} \in \sder_n$ the divergence condition for $\tilde{u}$ is equivalent to the commutativity of $\rho(\tilde{u})$ with $\delta^f_{\rm gr}$.
Since $\sigma_{\rm gr}$ is surjective, one can write $\tilde{u} = \sigma_{\gr}(b)$ for some $b\in \tr_n$.
Let $a \in \tr_n$. 
Using the facts (ii) and (iv) we compute 
\begin{align*}
\delta^f_{\gr}(\rho(\tilde{u})(a)) - \rho(\tilde{u})(\delta^f_{\gr}(a)) &= {\rm Div} ( \sigma_{\gr}(\rho(\tilde{u})(a))) - \sigma_{\gr}(b)({\rm Div}(\sigma_{\gr}(a)) \\
&= {\rm Div}(\sigma_{\rm gr}([b,a])) - \sigma_{\gr}(b)({\rm Div }(\sigma_{\gr}(a))) \\ 
&= - \sigma_{\gr}(a)({\rm Div}(\sigma_{\gr}(b))).
\end{align*}
In the last line, we have used the $1$-cocycle property of ${\rm Div}$.
By the fact (iii), it follows that $\sigma_{\gr}(a)({\rm Div}(\sigma_{\gr}(b))) = 0$ if and only if $[a, {\rm div}(\sigma_{\gr}(b))] = 0$.
Hence, $\rho(\tilde{u})$ commutes with $\delta^f_{\gr}$ if and only if ${\rm div}(\tilde{u})$ lies in the center of the Lie algebra $\tr_n$.
We can conclude by the fact (v).
\end{proof}
\end{remark}

\section{Proof of the main result} \label{sec:pfmain}

In this section, we prove Theorem~\ref{thm:main} in the introduction.
When $n=2$, we use letters $x,y$ for generators of $\lie_2$ instead of $x_1, x_2$.

\subsection{KV equations from emergent associator equations} \label{subsec:KV_eaq}

We prove the first statement of Theorem~\ref{thm:main}.

We simply write $u = \rho(\tilde{u})$ for $\tilde{u} \in \tder_n$.
The action of $\tilde{u}$ on tensor products of $\ass_n$ and $\tr_n$ is abbreviated in a similar way.
For instance, $\tilde{u}$ acts on $\tr_n \otimes \ass_n$ as $ \rho(\tilde{u}) \otimes {\rm id} + {\rm id} \otimes \rho(\tilde{u})$, and we denote it by $u$.

\begin{lem} \label{lem:mu_rho}
Let $\tilde{u} \in {\sf sder}_n$. Then, 
$d_{\tilde{u}} := \mu^f_{\rm gr} \circ u - u \circ \mu^f_{\rm gr}$ is a derivation on ${\sf ass}_n$.
Furthermore, the map $\mathcal{D}_{\tilde{u}}: = \mu^f_{r,\gr} \circ u - u \circ \mu^f_{r, \gr}$ from $\ass_n$ to $\tr_n \otimes \ass_n$ satisfies the following property: for any $a, b \in \ass_n$, 
\[
\mathcal{D}_{\tilde{u}}(ab) = \mathcal{D}_{\tilde{u}}(a) (1\otimes b) + (1 \otimes a) \mathcal{D}_{\tilde{u}}(b).
\]
\end{lem}

\begin{proof}
We simply write $\mu = \mu^f_\gr$ and $\eta = \eta_\gr$.
Let $a, b \in \ass_n$.
We compute
\begin{align*}
\mu (u(ab)) 
&= \mu( u(a)b  + a u(b)) \\
&= \mu( u(a)) b + u(a) \mu(b) + \eta(u(a), b) \\ 
& \quad + \mu(a) u(b) + a \mu(u(b)) + \eta(a, u(b)), \\
u( \mu(ab)) &=
u( \mu(a) b + a \mu(b) + \eta(a,b) \\
&= u(\mu(a)) b + \mu(a) u(b) + u(a) \mu(b) + a u(\mu(b)) + u(\eta(a,b)).
\end{align*}
Since $\tilde{u} \in \sder_n$, we have $\eta(u(a), b) + \eta(a, u(b)) = u(\eta(a,b))$ by Proposition~\ref{prop:eta_sder}.
Hence we see that $d_{\tilde{u}}$ is a derivation on $\ass_n$.

The map $\mu^f_{r, \gr}$ decomposes as shown in \eqref{eq:mufr_gr}.
Since the derivation $u$ commutes with the Hopf algebra operations on $\ass_n$, the second assertion follows from the first assertion.
\end{proof}

\begin{remark} \label{rem:d_u}
In fact, the derivation $d_{\tilde{u}} = \mu^f_\gr \circ u - u \circ \mu^f_\gr$ in Lemma~\ref{lem:mu_rho} can be written as 
\begin{equation} \label{eq:d_u_div}
d_{\tilde{u}} = \sigma_{\gr}({\rm div}(\tilde{u})),
\end{equation}
where $\sigma_\gr$ is the operation mentioned in Remark~\ref{rem:review_AKKN}.
This formula is the linearized version of \cite[eq.\,(84)]{AKKN_hg}.
\end{remark}

\begin{prop} \label{prop:u_delta}
Let $\tilde{u} \in {\sf sder}_n$ and assume that there is some $c\in {\sf ass}_n$ such that $\mu^f_{\rm gr}(u(x_i)) = [x_i, c]$ for all $i = 1,\ldots, n$.
Then, $\tilde{u} \in \krv_n$.
\end{prop}

\begin{proof}
By Theorem~\ref{thm:sderkrv}, it is enough to prove that $u$ commutes with $\delta^f_{\rm gr}$.  
A straightforward computation using \eqref{eq:mufr_gr} yields $\mu^f_{r, \gr}(u(x_i)) = |\iota(c')| \otimes [x_i, c'']$ for all $i=1,\ldots,n$, where we write $\Delta(c) = c' \otimes c''$ using the Sweedler notation.

Let $a = a_1\cdots a_m \in \ass_n$ be a product of $m$ elements of degree $1$.
Note that $\mathcal{D}_{\tilde{u}}(a_i) = |\iota(c')| \otimes [a_i, c'']$ since $\mu^f_\gr(a_i) = 0$.
By Lemma~\ref{lem:mu_rho}, we have
\begin{align*}
\mathcal{D}_{\tilde{u}}(a) 
& = \sum_{i=1}^m (1\otimes a_1 \cdots a_{i-1}) \mathcal{D}_{\tilde{u}}(a_i) (1\otimes a_{i+1} \cdots a_m) \\
& = \sum_{i=1}^m |\iota(c')| \otimes a_1\cdots a_{i-1} [a_i, c''] a_{i+1} \cdots a_m \\
& = |\iota(c')| \otimes [a, c''].
\end{align*}
Since $| [a, c''] | =0$, we obtain $(\delta^f_\gr \circ u - u \circ \delta^f_\gr)(|a|) = 0$ by \eqref{eq:delta_mu_gr}.
This completes the proof.
\end{proof}

\begin{proof}[Proof of Theorem~\ref{thm:main} (i)]
Let $\varphi = \varphi(x,y) \in \grt_1^\EM$.
Then $\varphi$ satisfies equation \eqref{eq:ppss5-gon} and the tangential derivation $\nu^\EM(\varphi) = (\varphi(y,x), \varphi(x,y))$ is special.
In particular, by Proposition~\ref{prop:eta_sder} we have
\begin{equation} \label{eq:payypa}
\pa_y \varphi = \pa^y \varphi = \iota(\pa_y \varphi).
\end{equation}

Put $f(s) := -(\pa_y \varphi)(s,0) \in \Q[[s]]$.
We will show that $\nu^\EM(\varphi) \in \sder_2$ satisfies the assumption of Proposition~\ref{prop:u_delta}.
By Proposition~\ref{prop:Rmugr}, one may replace $\mu^f_\gr$ with $R$.
We first compute
\begin{align*}
R(\nu^\EM(\varphi)(y)) &=
R([y, \varphi(x,y)]) \\
&= [y,R(\varphi)] + (\pa_y \varphi)y - y \iota(\pa_y \varphi) \\
&= [y, R(\varphi) - \pa_y \varphi] \\
&= [y, (\pa_y \varphi)(y,0) - (\pa_y \varphi)(x+y,0)] \\
&= [y, f(x+y)].
\end{align*}
Here, we have used formula \eqref{eq:Rrecursive} in the second line, equation~\eqref{eq:payypa} in the third line, equation \eqref{eq:ppss5-gon} in the fourth line, and the fact that $y$ commutes with any power series in $y$ in the last line.
Similarly, we compute 
\begin{align*}
R(\nu^\EM(\varphi)(x)) &=
R([x, \varphi(y,x)]) \\ 
&= [x, R(\varphi(y,x))] + \pa_x(\varphi(y,x)) x - x \iota(\pa_x(\varphi(y,x))) \\
&= [x, R(\varphi)(y,x) - (\pa_y \varphi)(y,x) ] \\
&= [x, (\pa_y \varphi)(x,0) - (\pa_y \varphi)(y+x,0)] \\
&= [x, f(x+y)].
\end{align*}
This completes the proof.
\end{proof}

\subsection{Symmetric Kashiwara-Vergne Lie algebra} \label{subse:sym_KV}

Recall from \cite[Section 8]{AT12} that the symmetric part of the Kashiwara-Vergne Lie algebra $\krv_2^{\rm sym}$ is the invariant Lie subalgebra of $\krv_2$ by the involution $(u(x,y), v(x,y)) \mapsto (v(y,x), u(y,x))$.
In this section, we prove the second statement of Theorem~\ref{thm:main}.

\begin{lem} \label{lem:put0}
Let $\varphi = \varphi(x,y) \in \lie_2$ be an element of degree at least two.
Then, $R(\varphi)(0,y) = R(\varphi)(x,0) = (\pa_y \varphi)(0,y) = 0$.
\end{lem}

\begin{proof}
Notice that $\varphi$ seen as an element of $\ass_2$ is a linear combination of words which contain at least one $x$ and at least one $y$.
Formula~\eqref{eq:mufgr} implies that $R(\varphi)$ is a linear combination of words with the same property.
Hence $R(\varphi)(0,y) = R(\varphi)(x,0) = 0$.
Similarly we have $(\pa_y \varphi)(0,y) = 0$, since $\pa_y \varphi$ is a linear combination of words which contain at least one $x$.
\end{proof}

\begin{proof}[Proof of Theorem~\ref{thm:main} (ii)]
Let $\tilde{u} = (\varphi(y,x), \varphi(x,y))\in \krv_2^{\rm sym}$ be homogeneous of degree at least two.

{\em Step 1.}
We first consider the case where $\tilde{u} \in \krv_2^0$.
By Theorem~\ref{thm:sderkrv} (ii), $\tilde{u}$ commutes with $\mu^f_\gr = R$.
Hence
\[
0 = \tilde{u}(R(y)) = R(\tilde{u}(y)) 
= R([y,\varphi]) 
= [y, R(\varphi) - \pa_y \varphi].
\]
Therefore, we have $R(\varphi) - \pa_y \varphi \in \Q[[y]]_{\ge 1}$.
By Lemma~\ref{lem:put0}, we obtain $R(\varphi) - \pa_y \varphi = 0$.
Furthermore, $(\pa_y \varphi)(x,0) = R(\varphi)(x,0) = 0$. 
Therefore, we obtain equation~\eqref{eq:ppss5-gon} for $\varphi$.
Hence $\varphi \in \grt_1^\EM$ and $\tilde{u} = \nu^\EM(\varphi)$.

{\em Step 2.}
We next consider the general case.
Let $k = \deg \varphi$.
If $k$ is even, then ${\rm div}(\tilde{u}) = 0$ by \cite[Proposition~4.5]{AT12}.
Hence $\tilde{u} \in \krv_2^0$, and $\tilde{u}$ is in the image of $\nu^\EM$  by Step 1.
Assume that $k$ is odd (and $\ge 3$).
There exists an element $\sigma_k \in \grt_1$ with the property
\[
{\rm div}(\nu(\sigma_k)) = 
|x^k + y^k - (x+y)^k |
\]
(see \cite[Proposition~6.3]{Drinfeld} and \cite[Proposition~4.10]{AT12}).
Thus there exists a constant $c \in \Q$ such that $\tilde{u} - c \nu(\sigma_k)$ has the vanishing divergence, i.e., $\tilde{u} - c \nu(\sigma_k)\in \krv_2^0$.
From Step 1, we obtain that $\tilde{u} - c \nu(\sigma_k)$ is in the image of $\nu^\EM$.
Let $\psi_k(x,y) = \sigma_k(-x-y, y) \in \grt_1^\EM$. 
Then $\nu(\sigma_k) = \nu^\EM(\psi_k)$.
Therefore, $\tilde{u} = (\tilde{u} - c \nu(\sigma_k)) + c \nu(\sigma_k)$ is in the image of $\nu^\EM$.
\end{proof}

\begin{remark} \label{rem:grt_1^em_0}
Let $\varphi \in \lie_2$ be an element of degree at least two.
The above proof shows that $\nu^\EM(\varphi) \in \krv_2^0$ if and only if $\varphi$ satisfies \eqref{eq:vscondition} and $R(\varphi) - \pa_y \varphi = 0$.
When this is the case, the element
\[
{\rm div}(\nu^\EM(\varphi)) = 
| x \pa_x(\varphi(y,x)) + y (\pa_y \varphi) | = 
|x (R(\varphi)(y,x)) + y R(\varphi)| 
\]
vanishes.
\end{remark}

\appendix

\section*{Appendix: Emergent knotted objects}
\label{subsec:eko}

This appendix was written by Dror Bar-Natan.

Consider some space $\calK$ of knotted objects in a three manifold $M$ (knots, or links, or knotted graphs, or tangles if $M$ has a boundary, or braids if one can specify a ``vertical'' direction in $M$).
If we mod out $\calK$ by the relation $\npc = \nnc$, knottedness is eliminated and what remains is a space of curves (or multi-curves, or graphs, etc.) regarded up to homotopy.
``Emergent knotted objects'' are what you get when you mod out $\calK$ by a slightly weaker version of $\npc = \nnc$ known in the theory of finite type invariants as the relation $\ndp \, \ndp =0$. 
Knottedness almost entirely disappears yet in the space we get, $\calK/(\ndp \, \ndp)$, we see that just a bit of knot theory begins to emerge beyond the homotopy theory that is already there.

More precisely, we consider $\bbQ\calK$, the space of $\bbQ$-linear combinations of (some type of) knotted objects in $M$.
As is often done when discussing finite type invariants~\cite{Bar-Natan:OnVassiliev, ChmutovDuzhinMostovoy:Vassiliev}, we extend $\calK$ by allowing double points ($\ndp$) while declaring that a double point is just a short for the linear combination $\npc - \nnc$; namely, we set $\ndp = \npc - \nnc$.
Then, as stated before, $\calK/\ndp$ is the space of (linear combinations of) curves modulo homotopy.
We continue as in the theory of finite type invariants and also study the quotient $\calK^{em}$ of $\bbQ\calK$ by knots that have {\it two} double points, where each of them represents a combination $\npc - \nnc$.
The space $\calK^{em}$ is the space of emergent knotted objects in $M$.

If $M$ is the Euclidean space $\bbR^3$ (or a ball in $\bbR^3$ when discussing tangles), homotopy theory is trivial so $\calK/\ndp$ becomes trivial (if, say, we are studying $n$-component links, they all become equivalent in $\calK/\ndp$).
The space of emergent knots $\calK^{em}=\calK/(\ndp \, \ndp)$ is slightly more interesting, though one may easily check that two links become equivalent if and only if all of the pairwise linking numbers of their components are the same (with similar statements for other types of knotted objects).
Thus in that case, $\calK^{em}$ is not completely trivial yet nevertheless quite simple.

In~\cite{LesDiablerets-2208, BNDHLS}, Bar-Natan, Dancso, Hogan, Liu, and Scherich study emergent tangles in a pole dancing studio $\PDS_m$ (an $m$-punctured disk cross an interval, or a room with $m$ vertical lines removed).
They find that in that case, an appropriate theory of expansions for emergent tangles leads to ``homomorphic'' expansions of the Goldman-Turaev Lie bialgebra (which in itself is a Lie bialgebra of curves modulo homotopy, a $(\calK/\ndp)$-type object).

The subquotients $\edk_{m,n}$ of the Drinfeld-Kohno Lie algebra $\dk_{m+n}$ that we use in this paper, in particular within the statement of our main theorem in Section~\ref{sec:intro_main}, arise from emergent braids in a pole dancing studio. 
Let us explain how.

Let $P_n$ denote the $n$-strand pure braid group, and let $P_{m,n}$ denote the kernel of the ``forget the last $n$ strands'' map $P_{m+n}\to P_m$; elements of $P_{m,n}$ are pure braids whose first $m$ strands remain stationary, like $m$ poles in a pole dancing studio (and hence we cease to call these first $m$ strands ``strands'', and instead refer to them as ``poles'').
Hence elements of $P_{m,n}$ can be thought of as $n$-strand braids in $\PDS_m$. 
Thus, following the above discussion of emergent knotted objects, we define $\Q P_{m,n}^\EM$ to be the quotient of $\Q P_{m,n}$ by the relation $\ndp \, \ndp = 0$, where all the strands involved in the double points belong to the last $n$ strands in $P_{m,n}$; namely, they are ``strands'' rather than ``poles''.

For the more algebraically-inclined, here is a fully-algebraic definition of $\Q P_{m,n}^\EM$: We let $\psi\colon P_n\to P_{m,n}\subset P_{m+n}$ be map which sends an $n$-strand pure braid into an $(m+n)$-strand pure braid by adding $m$ straight strands (``poles'') on the left; the resulting braid is clearly in $P_{m,n}$.
Let $I P_n$ be the augmentation ideal of $\Q P_n$ (namely, those $\bbQ$-linear combinations of elements of $P_n$ whose sum of coefficients is $0$).
Let $J$ be the two-sided ideal of $\Q P_{m,n}$ generated by $\psi(I P_n)$ (it is the two-sided ideal generated by $\ndp$, if both strands of the double point are strands and not poles). Finally, $\Q P_{m,n}^\EM:= \bbQ P_{m,n}/J^2$.
The augmentation ideal of $\Q P_{m,n}$ descends to a two-sided ideal of $\Q P_{m,n}^\EM$.

In complete generality, let $N$ be an associative $\Q$-algebra and $I$ a two-sided ideal of $N$.
For example, if $G$ is any group, one can set $N$ to be the group algebra $\Q G$ and $I = IG$ its augmentation ideal.
The powers of $I$ defines a decreasing filtration of $N$ and we can consider the completed associated graded algebra $\mathcal{A}_N$ defined by that filtration: $\mathcal{A}_N := \prod_{k\geq 0} I^k/ I^{k+1}$.
It is well known (see e.g.~\cite{Drinfeld, Kohno}) that $\calA_{\Q P_{m+n}}$ is the completed universal enveloping algebra of the Drinfeld-Kohno Lie algebra $\dk_{m+n}$ that we alluded to before, whose generators are $\{t_{ij}=[\sigma_{ij}-1]\colon 1\leq i\neq j\leq m+n\}$, where $\sigma_{ij}$ is the generator of $P_{m+n}$ in which strands $\#i$ and $\#j$ twist around each other once, in the positive direction.

It is now easy to determine $\calA_{\Q P_{m,n}^\EM}$: it is the completed universal enveloping algebra of the subquotient of $\dk_{m,n}$ in which all the generators of the form $\{t_{ij}\colon i,j\leq m\}$ are removed (as the generators $\{\sigma_{ij}\colon i,j\leq m\}$, the pole-pole twists, are removed from $P_{m,n}$), and in which we mod out by the square of the ideal generated by $\{t_{ij}\colon i,j>m\}$, corresponding to the emergent quotient $\ndp \, \ndp=0$. Thus $\calA_{\Q P_{m,n}^\EM} \cong U(\edk_{m,n})$ with the same $\edk_{m,n}$ as in the previous section. 
An elementary proof of this isomorphism is given in Proposition~\ref{prop:grpuremb2}.


\end{document}